\documentclass[final,leqno,onefignum,onetabnum]{siamltex1213}
\usepackage{dsfont}
\usepackage{amssymb}
\usepackage{amsfonts}
\usepackage{amsmath}
\usepackage{verbatim}
\usepackage{a4wide}
\usepackage{bm}
\usepackage{color}
\usepackage{appendix}
\usepackage{amssymb}
\usepackage{bbm}
\usepackage{fancybox}
\usepackage{fancyhdr}
\usepackage{graphicx}
\usepackage{mathrsfs} 
\usepackage{eufrak}     
\usepackage{mathrsfs} 
\usepackage{bbold}    
\usepackage{yhmath}
\usepackage{graphics}
\usepackage{enumerate}
\usepackage{xspace}
\usepackage{yhmath}
\usepackage{graphics}
\usepackage{enumerate}
\usepackage{xspace}
\usepackage{yhmath}
\usepackage{graphics}
\usepackage{enumerate}
\usepackage{xspace}
\usepackage{xcolor}
\usepackage{boxedminipage}
\usepackage{accents}
\usepackage{ifthen}
\usepackage{amsmath,amssymb}    
\usepackage{stmaryrd} 
\usepackage{eufrak}     
\usepackage{mathrsfs} 
\usepackage{bbold}    
\provideboolean{usemathrsfs}
\setboolean{usemathrsfs}{true}
\newcommand{\mathscript}{\mathscr}

\usepackage{enumerate}
\usepackage{xspace}
\usepackage{boxedminipage}
\usepackage[latin1]{inputenc}
\usepackage{verbatim}
\usepackage{color}
\usepackage[font=footnotesize]{subfig}
\usepackage{xspace}

\newcommand{\del}{\partial}

\renewcommand{\vec}[1]{\geovec{#1}}
\newcommand{\hatu}{\widehat{\vec u}}
\newcommand{\hatf}{\widehat{\vec f}}
\newcommand{\shatf}{\widehat{f}}
\newcommand{\shatu}{\widehat{u}}

\newcommand{\Cetad}{C_{\underline{\eta}}}
\newcommand{\Cetau}{C_{\overline{\eta}}}
\newcommand{\Cfu}{C_{\overline{\vec f}}}

\newcommand{\cE}{\ensuremath{\mathscript E}\xspace}

\newcommand{\hP}{\ensuremath{\mathcal P}\xspace}

\newcommand{\rN}{\ensuremath{\mathbb N}\xspace}
\newcommand{\rP}{\ensuremath{\mathbb P}\xspace}

\newcommand{\rR}{\ensuremath{\mathbb R}\xspace}
\newcommand{\rV}{\ensuremath{\mathbb V}\xspace}

\newcommand{\reals}{\rR}

\newcommand{\W}{\ensuremath{\Omega}\xspace}

\newcommand{\qp}[1]{\ensuremath{\left({#1}\right)}}

\newcommand{\qpbigg}[1]{\ensuremath{\bigg(#1\bigg)}}

\newcommand{\qb}[1]{\ensuremath{\!\left[{#1}\right]}}

\newcommand{\powqp}[2]{\ensuremath{\qp{#2}^{\kern -.2em\lower .3ex\hbox{\scriptsize $#1$}}\kern-.3em}}

\newcommand{\norm}[1]{\ensuremath{\left|#1\right|}}

\newcommand{\abs}[1]{\ensuremath{\left|#1\right|}}
\newcommand{\Norm}[1]{\ensuremath{\left\|#1\right\|}}

\newcommand{\ent}[3]{\ensuremath{#1(#2 \,\vert\,#3) }}

\newcommand{\ensemble}[2]{\ensuremath{\left\{ #1:\;#2 \right\}}}

\newcommand{\jump}[1]{\ensuremath{\left\llbracket #1\right\rrbracket}}

\renewcommand{\d}{\ensuremath{\,\mathrm{d}}}

 \providecommand{\D}{\ensuremath{\mathrm{D}}}

\newcommand{\registered}{
  \ensuremath{{}^{\bigcirc\!\;\!\!\!\!\!\!\!\;\text{\sc r}}}}

\newcommand{\AND}{\ensuremath{\text{ and }}}

\newcommand{\Oh} {\operatorname{O}}                   
\newcommand{\pd}[2]{\ensuremath{\partial_{#1}{#2}}\xspace} 

\newcommand{\pdt}{\pd t {}}                       
\newcommand{\pdx}{\pd x {}}

\newcommand{\transpose}{{\boldsymbol\intercal}}   
\newcommand{\Transpose}[1]{\ensuremath{{#1}^{\transpose}}}

\newcommand{\Hess}{\ensuremath{\D^2}}

\renewcommand{\vec}[1]{\ensuremath{\boldsymbol{#1}}}
\newcommand{\geovec}[1]{\ensuremath{\boldsymbol{#1}}}

\newcommand{\CC}{\ensuremath{\operatorname C}\xspace}

\newcommand{\LL}{\ensuremath{\operatorname L}\xspace}

\newcommand{\WW}{\ensuremath{\operatorname W}\xspace}

\newcommand{\cont}[1]{\ensuremath{\CC^{#1}}}

\newcommand{\leb}[1]{\ensuremath{\LL_{#1}}}

\newcommand{\sob}[2]{\ensuremath{\WW^{#1}_{#2}}}

\newcommand{\poly}[1]{\ensuremath{\rP}^{#1}}

\newcommand{\fes}[1]{\ensuremath{\fespace^{#1}}}

\newcommand{\EOC}{\ensuremath{\operatorname{EOC}}\xspace}

\newcommand{\Forall}{\:\forall\:}

\newcommand{\Foreach}{\quad\Forall}

\newcommand{\downto}{\ensuremath{\searrow}}

\newcommand{\Program}[1]{\textsf{#1}\xspace}

\newcommand{\matlab}{\Program{Matlab\registered}}

\providecommand{\ListParameters}{}
\renewcommand{\ListParameters}{
  \setlength{\topsep}{0em}
  \setlength{\leftmargin}{0em}
  \setlength{\itemsep}{0ex}
  \setlength{\parsep}{.5ex}
  \setlength{\itemindent}{\labelsep}
  \addtolength{\itemindent}{\labelwidth}
}

\newcounter{LetterListItem}
\renewcommand{\theLetterListItem}{(\alph{LetterListItem})}

\newcounter{NumberListItem}
\renewcommand{\theNumberListItem}{\arabic{NumberListItem}}

\newcounter{QuestionListItem}
\renewcommand{\theQuestionListItem}{\textbf{Question \arabic{QuestionListItem}}}

\newcounter{RomanListItem}
\renewcommand{\theRomanListItem}{(\roman{RomanListItem})}

\newcounter{StepsItem}

\providecommand{\ListParameters}{}
\renewcommand{\ListParameters} {
  \setlength{\topsep}{0em}
  \setlength{\leftmargin}{0em}
  \setlength{\itemsep}{0ex}
  \setlength{\parsep}{.5ex}
  \setlength{\itemindent}{\labelsep}
  \addtolength{\itemindent}{\labelwidth}
}
\newcommand{\Proofname}{Proof}

\newcommand{\pdfformat}[1]{
  \provideboolean{pdfoutput}
  \setboolean{pdfoutput}{#1}
  \ifthenelse{\boolean{pdfoutput}}%
	     {\typeout{using pdf}
               \usepackage{pdfsync}
	       \usepackage[pdftex]{graphicx,xcolor}
	       \newcommand{\graphext}{pdf}
	       \newcommand{\graphextex}{pdf_t}
	       \usepackage{epsfig}
	       \usepackage{tikz}
	       \usepackage{rotating}
	     }
	     {
	       \typeout{using eps}
	       \usepackage[dvips]{graphicx,xcolor}
	       \newcommand{\graphext}{eps}
	       \newcommand{\graphextex}{eps_t}
	       \usepackage{epsfig}
	       \usepackage{tikz}
	       \usepackage{rotating}
	     }
}
\newcommand{\tn}{{t_n}}
\renewcommand{\geovec}[1]{\boldsymbol{#1}}
\renewcommand{\vec}[1]{\ensuremath{\boldsymbol
    {\mathsf{#1}}}}

	       {\par\noindent{\bf \Proofname\ #1}}
	       {{\raggedright{{ }\hfill\qed}}} 
    
\newtheorem{The}[subsection]{Theorem}
\newtheorem{Pro}[subsection]{Proposition}
\newtheorem{Pro*}{Proposition}
\newtheorem{Lem}[subsection]{Lemma}

\newtheorem{Rem}[subsection]{Remark}
\newtheorem{Defn}[subsection]{Definition}
\newtheorem{Hyp}[subsection]{Assumption}

\newcommand{\V}[1]{\fes^{#1}}

\newcommand{\EI}{\ensuremath{\operatorname{EI}}\xspace}

\newcommand{\ie}{i.e.,\xspace}

\newcommand{\eint}[3]{\qpbigg{\int_{#1}^{#2}#3^2}^{\kern -.2em\lower .3ex\hbox{\scriptsize $1/2$}}\kern-.3em}

\renewcommand{\fes}{\ensuremath{\mathbb V}}

\renewcommand{\fes}{\ensuremath{\mathbb V}}
\renewcommand{\sin}[1]{\ensuremath{\operatorname{sin}\left(#1\right)}}

\newcommand{\avg}[1]{\lbrace #1 \rbrace}

\newcommand{\doublelbrace}{
  \mathrel{\ooalign{\cr\kern+1.0pt$\lbrace$\cr\kern-0.5pt$\lbrace$}}}
\newcommand{\doublerbrace}{
  \mathrel{\ooalign{$\rbrace$\cr\kern-1.5pt$\rbrace$\cr\kern-1.0pt}}}

\renewcommand{\avg}[1]{\doublelbrace #1 \doublerbrace}

\providecommand{\qb}[1]{\ensuremath{\!\left[{#1}\right]}}

\title
{A posteriori analysis of  discontinuous Galerkin schemes for systems of hyperbolic conservation laws}

\author{
  Jan Giesselmann\footnotemark[2]
  \
  \footnotemark[5]
  \
  \footnotemark[6]
  \and
  Charalambos Makridakis\footnotemark[3]
  \
  \footnotemark[5]
  \and
  Tristan Pryer\footnotemark[4]
  \
  \footnotemark[5]
  \
  \footnotemark[7]
}

\date{\today}

\begin{document}

\maketitle

\renewcommand{\thefootnote}{\fnsymbol{footnote}}

\footnotetext[2]
             {Institute of Applied Analysis and Numerical Simulation,
               University of Stuttgart,
               Pfaffenwaldring 57,
               D-70563 Stuttgart,
               Germany
               (\email{{jan.giesselmann@mathematik.uni-stuttgart.de}})
             }
     
\footnotetext[3]
             {
               School of Mathematical and Physical Sciences,
               University of Sussex,
               Brighton, 
               BN1 9QH, UK,
               (\email{{C.Makridakis@sussex.ac.uk}})
             }
             
\footnotetext[4]
             {
               Department of Mathematics and Statistics,
               Whiteknights,
               PO Box 220,
               Reading,
               RG6 6AX,
               England UK,
               (\email{{T.Pryer@Reading.ac.uk}} )
             }

\footnotetext[5]
             {
               The authors were supported by the  the FP7-REGPOT project ``ACMAC--Archimedes Center for Modeling, Analysis and Computations'' of the University of Crete (FP7-REGPOT-2009-1-245749).
             }
\footnotetext[6]
             { 
               J.G. was also partially supported by  the 
               German Research Foundation (DFG) via SFB TRR 75 ``Tropfendynamische Prozesse unter extremen Umgebungsbedingungen''.
             }

\footnotetext[7]
             {
               T.P. was also partially supported by the EPSRC grant EP/H024018/1 and an LMS travel grant 41214.
             }

\renewcommand{\thefootnote}{\arabic{footnote}}

\begin{abstract}
  In this work we construct reliable a posteriori estimates for some
  discontinuous Galerkin schemes applied to nonlinear
  systems of hyperbolic conservation laws. We make
  use of appropriate \emph{reconstructions} of the discrete solution
  together with the \emph{relative entropy} stability framework.
  
  The methodology we use is quite general and allows for a posteriori
  control of discontinuous Galerkin schemes with standard flux choices
  which appear in the approximation of conservation laws.

  In addition to the analysis, we conduct some numerical benchmarking 
  to test the robustness of the resultant estimator.
\end{abstract}

\section{Introduction}
Hyperbolic conservation laws play an important role in many physical
and engineering applications.  One example is the description of
non-viscous compressible flows by the Euler equations.  Hyperbolic
conservation laws in general only have smooth solutions up to some
finite time even for smooth initial data.  This makes their analysis
and the construction of reliable numerical schemes challenging.  The development of discontinuities poses 
significant challenges to their numerical simulation. Several successful schemes were developed so far 
and are mainly based on 
finite differences,   finite volume and discontinuous Galerkin (dG) finite
element schemes.  For an overview on these schemes we refer to
\cite{GR96,Kro97,LeV02, Cockburn2003, HW08} and their references.  
In this work we are interested in a posteriori error control of hyperbolic systems while solutions 
are still smooth.  Our main tools are appropriate reconstructions of the discontinuous Galerkin schemes considered
and relative entropy estimates. 

The first systematic a posteriori analysis for numerical approximations of scalar conservation laws accompanied with corresponding adaptive algorithms, can be traced back to 
\cite{KO00, GM2000}, see also \cite{Cockburn2003, DMO07} and their references. These estimates were derived by employing 
Kruzkov's estimates. A posteriori results for systems were derived in \cite{Laforest2008,Laforest2004} for  front tracking and Glimm's schemes, see also
\cite{KLY2010}. For recent a posteriori analysis for well balanced schemes   for a damped semilinear wave equation we refer to \cite{AmadoriGosse13}.

We aim at providing a rigorous a posteriori error estimate for
semidiscrete dG schemes applied to systems of hyperbolic conservation
laws which are of optimal order. The extension of these results to
fully discrete schemes is obviously an important point but exceeds the
scope of the work at hand.  Our analysis is based on an extension of
the reconstruction technique, developed mainly for 
discretisations of parabolic problems, see \cite{Mak07} and references
therein, to space discretisations in the hyperbolic setting. The main
idea of the reconstruction technique is to introduce an
intermediate function, which we will denote $\widehat u$, which solves
a \emph{perturbed} partial differential equation (PDE). This perturbed
PDE is constructed in such a way that this $\widehat u$ is sufficiently
close to both the approximate solution, denoted $u_h$ and the exact
solution to the conservation law, denoted $u$. Then,  typically 
\begin{equation}
  \Norm{u - u_h} \leq \Norm{u - \widehat u} + \Norm{\widehat u - u_h},
\end{equation}
 where $\Norm{\widehat u - u_h}$ can be controlled explicitly and $ \Norm{u - \widehat u}  $ is estimated  
  using perturbation  stability  techniques.
%
 For systems of hyperbolic conservation laws admitting a convex entropy 
 the \emph{relative entropy} technique, introduced in \cite{Daf79,Dip79} provides a  natural
stability framework in the case where one of the two  functions involved in the analysis is a Lipschitz solution of the conservation law. 
This technique is based on the fact that usually systems of hyperbolic
 conservation laws are endowed with an 
 entropy/entropy flux pair. For conservation laws describing physical
 systems this notion of entropy follows from the physical one. 
The entropy/entropy flux pair gives rise to an admissibility condition
 for weak solutions, cf. Definition \ref{def:esol}, which leads to the notion
 of entropy solutions. 
It can also be used to define the notion of relative entropy between two solutions.
In case of a convex entropy the relative entropy can be used to control 
the $\leb{2}$ distance.
It can be used to obtain a stability result, Theorem \ref{lem:ltstab},
 which implies uniqueness of Lipschitz solutions in the class of entropy solutions.
One drawback of this stability framework is that a Gronwall type argument has to
be employed such that the error estimate depends exponentially on
time.  
There are two features of the relative entropy framework which need to be taken into account when
constructing the reconstruction $\hat u.$
If the relative entropy is to be used to compare $u, \hat u$ one of the two needs to be Lipschitz. 
As $u$ may be discontinuous, $\hat u$ needs to be Lipschitz.
Secondly, the relative entropy is an $\leb{2}$ framework, thus, the residuals in the perturbed equation
satisfied by $\hat u$ need to be in $\leb{2}$.


Relative entropy techniques for the a priori error  analysis of approximations of {systems} of conservation laws were first used in 
\cite{AMT2004}.
For other works concerning analysis of schemes for  systems of
conservation laws see, e.g.\ \cite{JR05,JR06}.  
For discontinuous
Galerkin/Runge--Kutta (dGRK) schemes a priori estimates can be found
in \cite{ZS04,ZS10}. 
 In \cite{HH02} the authors
 use  a goal oriented framework providing error indicators  for a
space-time dG scheme. These indicators are computable, provided  that certain  dual problems are well posed.  
Asymptotic nodal superconvergence is
investigated in a series of papers, see \cite{BA11} and references therein. In \cite{DMO07} the authors provide an
a posteriori estimate for the $\leb{1}$ error of dGRK schemes approximating a scalar conservation law, see also \cite{Ohl09}
for an overview on a posteriori error analysis for hyperbolic conservation laws.

The novelty of this work is that it provides a posteriori estimates for
dG schemes for  nonlinear \emph{systems} of conservation laws. Notice we do not assume  anything  
on the exact solution apart from the fact that it takes values on a compact set known a priori. 
That said,  the final estimate is    \emph{conditional,} i.e., holds under assumptions  on the approximation and its reconstruction, see \cite{MN06, Mak07}, which 
can be verified 
 a posteriori.  
It must be noted, however, that our estimates are essentially valid before the formation of shocks.  In the case where the entropy solution is discontinuous, our error estimator does not
converge to zero if the meshwidth goes to zero.
This is explained in detail in Remark \ref{rem:des} and is an expected   direct consequence of the fact that in the relative entropy framework the Lipschitz constant
of one of the solutions, which are compared to each other, enters the error estimate. The extension of our approach to the case 
of non-smooth solutions is a very challenging problem which is  currently under investigation.  The need of introducing reconstruction operators imposes some restrictions 
on the permitted discrete fluxes used in the dG method, see Remark 3.1. 
We present our
analysis  in  the one dimensional case.  An extension of our results
to several space dimensions would require a generalised reconstruction
technique while the other arguments would be analogous.

The remainder of this paper is organized as follows: In \S
\ref{sec:claw} we give some background on hyperbolic conservation laws
and their stability via the relative entropy method.  In \S
\ref{sec:scheme} we describe the numerical schemes under
consideration.  In \S \ref{sec:rec} we provide some background on
reconstruction methods and we discuss the reconstruction
procedure which we employ here and study its properties.  In \S \ref{sec:err} we combine the
reconstruction and the relative entropy methodology to derive an a
posteriori error estimate.  Finally, in \S \ref{sec:num} we show some
numerical experiments employing the estimates derived in \S
\ref{sec:err}, studying their asymptotic properties.

\section{Preliminaries, conservation laws and relative entropy}
\label{sec:claw}

Given the standard Lebesgue space notation
\cite{Cia78,Evans:1998} we begin by introducing the Sobolev
spaces. Let $\W \subset \reals$ then
\begin{gather}
  \sob{k}{p}(\W)
  := 
  \ensemble{\phi\in\leb{p}(\W)}
  {\D^{\alpha}\phi\in\leb{p}(\W), \text{ for } \norm{\alpha}\leq k},
\end{gather}
which are equipped with norms and seminorms
\begin{gather}
  \Norm{u}_{\sob{k}{p}(\W)}
  := 
  \begin{cases}
    \qp{\sum_{\norm{\alpha}\leq k}\Norm{\D^{\alpha} u}_{\leb{p}(\W)}^p}^{1/p} &\text{ if } p \in [1,\infty)
    \\
    \sum_{\norm{\alpha}\leq k}\Norm{\D^{\alpha} u}_{\leb{\infty}(\W)} &\text{ if } p = \infty 
  \end{cases}
  \\
  \norm{u}_{\sob{k}{p}(\W)}
  :=
  \Norm{\D^k u}_{\leb{p}(\W)} 
\end{gather}
respectively, where derivatives $\D^{\vec\alpha}$ are understood in a
weak sense. 

We use the convention that when derivatives act on a vector valued
multivariate function, $\geovec u = \Transpose{\qp{u_1,
    \dots, u_d}}$, it is meant componentwise, that is
$\partial_{x} \geovec u = \Transpose{\qp{\partial_x u_1,
    \dots, \partial_{x} u_d}}$ denotes a column vector. The derivative
of a field, $q$ say, with respect to the dependent variable is denoted
$\D q = {\qp{\partial_{{u_1}} q(\vec u), \dots, \partial_{{u_d}} q(\vec
    u)}}$ is a row vector. The matrix of second derivatives of $q$ is
\begin{equation}
  \Hess q(\vec u)
    :=
    \begin{bmatrix}
      \partial_{{u_1,u_1}} q(\vec u)
      ,\dots ,
      \partial_{{u_1,u_d}} q(\vec u)
      \\
      \vdots\qquad \ddots \qquad\vdots
      \\
      \partial_{{u_d,u_1}} q(\vec u)
      ,\dots ,
      \partial_{{u_d,u_d}} q(\vec u)
      \end{bmatrix}.
  \end{equation}
  For a vector field $\geovec f$, we denote its Jacobian by
  $\D{\geovec f}$ which is also a $d\times d$ matrix and its Hessian
  as $\Hess \vec f$ which is given as a $3$--tensor. We also make use
  of the following notation for time dependent Sobolev (Bochner)
  spaces:
\begin{equation}
  \leb{\infty}(0,T; \sob{k}{p}(\W))
  :=
  \ensemble{u : [0,T] \to \sob{k}{p}(\W)}
           {\sup_{t\in [0,T]} \Norm{u(t)}_{\sob{k}{p}(\W)} < \infty}.
\end{equation}

Let $U \subset \rR^d$ convex be the state space.
We consider the following
first order (system of) conservation law(s) 
\begin{equation}\label{eq:cl}
  \pdt \vec u(x,t) + \pdx \vec f(\vec u(x,t)) =0 \text{ for } (x,t) \in (0,1) \times (0,\infty).
\end{equation}
We complement (\ref{eq:cl}) with the following initial and boundary conditions
\begin{equation}\label{eq:ibc}
  \vec u(0,t)= \vec u(1,t) \text{ for } t \in  (0,\infty) \quad \text{and} \quad \vec u(x,0) = \vec u_0(x) \text{ for } x \in (0,1)
\end{equation}
for some function $\vec u_0 \in
\leb{\infty}((0,1),U).$ The solution, which in general is only in
$\leb{\infty}((0,1)\times (0,\infty), U )$, takes values in the state
space and we assume the flux function $\vec f: U \rightarrow \rR^d$ is
at least $\cont{2}(U)$.

In particular, in our estimates, the assumed regularity will depend on
the polynomial degree of the employed dG method.  Throughout this
paper we will assume that there is an entropy/entropy-flux pair
$(\eta,q)$ with $\eta \in C^2(U, \rR)$ strictly convex and $q \in
C^1(U,\rR)$ associated to \eqref{eq:cl} in such a way that
\begin{equation}\label{eflux}
 \D q 
 =
 \D \eta
 \D {\vec f}.
\end{equation}
The existence of an entropy flux implies that
\begin{equation}\label{commute}
  \Transpose{\qp{\D\vec f}}\D^2\eta = \D^2\eta \D\vec f.
\end{equation}
It is readily verifiable that strong solutions of \eqref{eq:cl}
satisfy the additional conservation law
\begin{equation}\label{eq:ecl}
  \pdt \eta(\vec u) + \pdx q(\vec u) =0.
\end{equation}

For general background on hyperbolic conservation laws the reader is
refered to \cite[c.f.]{Daf10,LeF02}.
Note that not every system of hyperbolic conservation laws admits a convex entropy/entropy flux pair,
 see \cite[Sec. 5.4]{Daf10},
even if it is physically meaningful.
The derivation of a posteriori error estimates for systems of hyperbolic conservation laws admitting
 only poly or quasiconvex entropies is  beyond the scope of this work.
 It is common that
solutions of \eqref{eq:cl} develop discontinuities after finite time.
This motivates developing a notion of \emph{weak solution}. As weak
solutions, which satisfy the equation in the distributional sense are
not unique attention is restricted to so
called 
\emph{entropy solutions} $\vec u \in \leb{\infty}((0,1)\times
(0,\infty),U)$.  The concept of entropy solution guarantees uniqueness
of solutions for scalar problems and can be interpreted as enforcing
that solutions are compatible with the 2nd law of thermodynamics.
However, it is important to note that entropy solutions need not be
unique for systems of conservation laws in multiple space dimensions
even if these are endowed with a convex entropy, \cite{DS10}.
In this context it should be noted that the relative entropy technique, see Lemma \ref{lem:ltstab}, guarantees 
uniqueness for entropy solutions if and only if they are Lipschitz.
  The
notion of entropy solution can be motivated by the \emph{vanishing
  viscosity} framework. Consider the regularised PDE
\begin{equation}
  \label{eq:visc-soln}
  \pdt \vec u^\epsilon + \pdx \vec f(\vec u^\epsilon) = \epsilon \pd{xx}{\vec u^\epsilon}.
\end{equation}
Inserting the solution of (\ref{eq:visc-soln}) into the conservation
law (\ref{eq:ecl}) we see
\begin{equation}
  \begin{split}
    \pdt \eta(\vec u^{\epsilon}) + \pdx q(\vec u^{\epsilon})
    &=
    \D
    \eta(\vec u^\epsilon)
    \pdt \vec u^\epsilon
    +
    \D
    q(\vec u^\epsilon)
    \pdx \vec u^\epsilon
    \\
    &=
    \D
    \eta(\vec u^\epsilon)
    \qp{\epsilon \pd{xx}{\vec u^\epsilon} - \pdx \vec f(\vec u^\epsilon)}
    +
    \D
    q(\vec u^\epsilon)
    \pdx \vec u^\epsilon
    \\
    &=
    \D
    \eta(\vec u^\epsilon)
    \qp{\epsilon \pd{xx}{\vec u^\epsilon}
      -
      \D \vec f(\vec u^\epsilon) \pdx \vec u^{\epsilon}
    }
    +
    \D
    q(\vec u^\epsilon)
    \pdx \vec u^\epsilon
    \\
    &=
    \D
    \eta(\vec u^\epsilon)
    \epsilon \pd{xx}{\vec u^\epsilon}
    \\
    &=
    \epsilon \pdx\qp{\D \eta(\vec u^\epsilon) \pdx \vec u^{\epsilon}}
    -
    \epsilon \qp{\D^2\eta(\vec u^\epsilon) \pdx \vec u^\epsilon} \pdx \vec u^\epsilon
    \\
    &\leq
    \epsilon \pdx\qp{\D \eta(\vec u^\epsilon) \pdx\vec u^{\epsilon}}.
  \end{split}
\end{equation}
The limit as $\epsilon\to 0$ yields the following definition
\begin{Defn}[entropy solution]\label{def:esol}
  A function $\vec u\in \leb{\infty}((0,1) \times [0,\infty),U)$ is
  said to be an \emph{entropy solution} of the initial boundary value
  problem (\ref{eq:cl})--\eqref{eq:ibc}, with associated
  entropy/entropy-flux pair $(\eta,q)$, if
  \begin{equation}\label{eq:wsol}
    \int_0^\infty 
    \int_0^1
    \vec u 
    \cdot 
    \pdt \vec \phi 
    +
    \vec f(\vec u) 
    \cdot
    \pdx \vec \phi \d x \d t 
   + 
    \int_0^1
    \vec u_0 
    \cdot 
    \vec \phi (\cdot,0)
    \d x 
    =
    0
    \Foreach
    \vec \phi \in \cont{\infty}_c(S^1\times [0,\infty),\rR^d)
  \end{equation}
  and 
  \begin{equation}\label{eq:esol}
    \int_0^\infty 
    \int_0^1 
    \eta(\vec u) \pdt \phi 
    +
    q(\vec u)\pdx  \phi \d x \d t 
    + 
      \int_0^1
    \eta(\vec u_0 )
    \phi (\cdot,0)
    \d x 
    \geq 0  \Foreach \phi \in C_c^\infty(S^1\times [0,\infty),[0,\infty)).
  \end{equation}
  Here $S^1$ (the 1--sphere) refers to the unit interval $[0,1]$ with
  matching endpoints.
\end{Defn}
\begin{Rem}[scalar case]
  In the scalar case entropy solutions are required to satsify \eqref{eq:esol} for every convex entropy/entropy flux pair.
\end{Rem}

For $\vec u\in \leb{\infty}((0,1) \times
(0,\infty),U)$ the distribution $\pdt\eta(\vec u) +\pdx q(\vec u)$ has
a sign and therefore is a measure, \ie we may replace the smooth test
functions in Definition \ref{def:esol} by Lipschitz continuous ones.
Stability of solutions and in particular uniqueness of Lipschitz
solutions within the class of entropy solutions is obtained via
relative entropy arguments, see \cite[Chapter 5]{Daf10} and references
therein.

\begin{Defn}[relative entropy and entropy-flux]
  We define the \emph{relative entropy}, $\ent{\eta}{\vec u}{\vec v}$,
  and \emph{relative entropy-flux}, $\ent{q}{\vec u}{\vec v}$, of two
  generic vector valued functions $\vec v \AND \vec w$ with values in $U$ to be
  \begin{equation}
    \label{def:relen}
    \begin{split}
    \ent{\eta}{\vec v}{\vec w} 
    &:=
    \eta(\vec v) 
    -
    \eta(\vec w)
    -
    \D \eta(\vec w) (\vec v-\vec w)
    \\
    \ent{q}{\vec v}{\vec w} 
    &:=
    q(\vec v)
    -
    q(\vec w) 
    -
    \D \eta(\vec w) (\vec f(\vec v)-\vec f(\vec w)).
  \end{split}
\end{equation} 
Note that $\ent{\eta}{\vec v}{\vec w} $ and $\ent{q}{\vec v}{\vec w}$
are not symmetric in $\vec v,\, \vec w.$
\end{Defn}

\begin{Hyp}[values in a compact set]
 We will assume throughout the paper that the exact solution $\vec u$ of \eqref{eq:cl} takes values in $\mathfrak{O},$\ie 
\[ u(x,t) \in \mathfrak{O} \quad \forall \ (x,t) \in (0,1) \times (0,\infty),\]
where $\mathfrak{O} $ be a compact and convex subset of $U$.
\end{Hyp}

\begin{Rem}[Bounds on flux and entropy]
 Due to the regularity of $\vec f$ and $\eta$ and the compactness of $\mathfrak{O}$ there are
  constants $0 < \Cfu < \infty$ and $0< \Cetad < \Cetau < \infty$ such
  that
  \begin{equation}
    \label{eq:consts}
\norm{\Transpose{\vec v} \Hess \vec f(\vec u) \vec v} 
    \leq \Cfu \norm{\vec v}^2,    
    \qquad 
 \Cetad 
    \norm{\vec v}^2
    \leq 
    \Transpose{\vec v} 
    \Hess \eta(\vec u)
    \vec v
    \leq \Cetau \norm{\vec v}^2
  \Foreach \vec v \in \reals^d, \vec u \in \mathfrak{O},
  \end{equation}
  where $\norm{\cdot}$ is the Euclidean norm for vectors.
  Note that $\Cfu$, $\Cetad$ and $\Cetau$ can be explicitly computed from $\mathfrak{O}$, $\vec f$ and $\eta.$
\end{Rem}

\begin{Lem}[Gronwall inequality]
  \label{lem:gronwall}
  {Given $T>0$, let $\phi(t)\in\cont{0}([0,T])$ and $a(t), b(t) \in
  \leb{1}([0,T])$ all be nonnegative functions with $b$ non-decreasing and satisfying
  \begin{equation}
 \phi(t) \leq \int_0^t a(s) \phi(s) \d s + b(t).
  \end{equation}}
  Then 
  \begin{equation}
    \phi(t) \leq b(t) \exp\qp{\int_0^t a(s) \d s} \Foreach t \in [0, T].
  \end{equation}
\end{Lem}
As we will make use of a similar argument to derive our error estimate
let us give the proof of the following stability result which can be
found in \cite{Daf10}.
\begin{Lem}[$\leb{2}$--stability]
  \label{lem:ltstab}
  Let $ \vec u$ be an entropy solution of \eqref{eq:cl}--\eqref{eq:ibc}
  corresponding to initial data $\vec u_0$ and $ \vec v$ a Lipschitz
  solution of \eqref{eq:cl}--\eqref{eq:ibc} corresponding to initial
  data $\vec v_0.$ Let $\vec u$ and $\vec v$ take values in $\mathfrak{O}.$
  Then there exist constants $C_1,C_2>0$ such that
  \begin{equation}
    \Norm{\vec u(\cdot,t) - \vec v (\cdot, t) }_{\leb{2}(I)}
    \leq 
    C_1 \exp( C_2 t) \Norm{ \vec u_0 - \vec v_0}_{\leb{2}(I)}.
  \end{equation}

\end{Lem}
\begin{proof}
  Note that $\vec v$ satisfies \eqref{eq:esol} as an
  equality. Thus, for any Lipschitz continuous, non negative test
  function $\phi$ we have
  \begin{equation}
    \int_0^\infty 
    \int_0^1 
    \pdt \phi ( \eta(\vec u) - \eta(\vec v) ) 
    +
    \pdx \phi ( q(\vec u) - q(\vec v) ) \d x \d t
    + 
    \int_0^1 \phi(\cdot,0)  \qp{\eta(\vec u_0) - \eta(\vec v_0)}
    \d x
    \geq 0.
\end{equation}
Using the definition of relative entropy and relative entropy flux, we
may reformulate this as
\begin{equation}
  \label{Daf1}
  \begin{split}
  \int_0^\infty 
  \int_0^1
  \pdt \phi \qp{ \ent{\eta}{\vec u}{\vec v}
    + \D \eta(\vec v) \qp{\vec u - \vec v} }
  +
  \pdx \phi &\qp{ 
    \ent{q}{\vec u}{\vec v}
    +
    \D\eta(\vec v) \qp{\vec f(\vec u) - \vec f(\vec v)}}
  \d x \d t
  \\
  & \qquad \qquad  \qquad + 
  \int_0^1
  \phi(\cdot,0) \qp{\eta(\vec u_0) - \eta(\vec v_0)} \d x
  \geq 0.
\end{split}
\end{equation}
Upon using the Lipschitz continuous test function $ \vec \phi = \phi
\D\eta(\vec v)$ in \eqref{eq:wsol} for $\vec u$ and $\vec v,$
we obtain
\begin{multline}\label{Daf2}
 \int_0^\infty \int_0^1 \pdt( \phi \D\eta(\vec v)) (\vec u - \vec v)  +
    \pdx (\phi  \D\eta(\vec v)) (\vec f(\vec u) - \vec f(\vec v))  \d x \d t
+ \int_0^1\phi(\cdot,0) \D\eta(\vec v(\cdot, 0))(  \vec u_0 - \vec v_0 )\d x
=0.
\end{multline}
We use the product rule in \eqref{Daf2} and combine it with \eqref{Daf1} to obtain
\begin{equation}\label{Daf3}
  \begin{split}
    \int_0^\infty \int_0^1 \pdt \phi \ent{\eta}{\vec u}{\vec v}   +
    \pdx \phi  \ent{q}{\vec u}{\vec v}   \d x \d t  
    -
    \int_0^\infty \int_0^1 \phi( \pdt \vec v \D^2\eta(\vec v) (\vec u -& \vec v) + \pdx \vec v \D^2\eta(\vec v) (\vec f(\vec u) - \vec f(\vec v))  ) \\
    &
    + \int_0^1\phi(\cdot,0)  \ent{\eta}{\vec u_0}{\vec v_0}  \d x
\geq 0.
  \end{split}
\end{equation}
Using $\pdt \vec v= -D\vec f(\vec v)\pdx \vec v$ and \eqref{commute} we  find
\begin{equation}\label{Daf4}
  \begin{split}
    \int_0^\infty \int_0^1 \pdt \phi \ent{\eta}{\vec u}{\vec v}
    +
    \pdx \phi  \ent{q}{\vec u}{\vec v}
    \d x \d t  -
    \int_0^\infty \int_0^1 \phi( \pdx \vec v 
    &\D^2\eta(\vec v) (\vec f(\vec u) - \vec f(\vec v)  - \D\vec f(\vec v) (\vec u - \vec v) ) \\
    &\qquad\qquad\qquad+ \int_0^1\phi(\cdot,0)  \ent{\eta}{\vec u_0}{\vec v_0}  \d x
    \geq 0.
  \end{split}
\end{equation}
Now we fix $t>0$. Then for every $0<s<t$ and $\varepsilon >0$ we consider the test function
\begin{equation}
  \label{eq:test-function-phi}
  \phi(x,\sigma) = \left\{ \begin{array}{lcl} 
    1 & : & \sigma <  s \\
    1 - \frac{\sigma-s}{\varepsilon} & : & s < \sigma < s+ \varepsilon\\
    0 &:& \sigma >s+ \varepsilon 
  \end{array}\right..
\end{equation}
In this case we infer from \eqref{Daf4}
\begin{equation}\label{Daf5}
  \begin{split}
    -\frac{1}{\varepsilon}
    \int_s^{s+\varepsilon} \int_0^1 \ent{\eta}{\vec u}{\vec v}    \d x \d t  
    -
    \int_0^\infty \int_0^1 \phi( \pdx \vec v \D^2\eta(\vec v) (\vec f(\vec u) - \vec f(\vec v)  
    - &\D\vec f(\vec v) (\vec u - \vec v)) ) \d x \d t
    \\
&\qquad\qquad + \int_0^1 \ent{\eta}{\vec u_0}{\vec v_0}  \d x
\geq 0.
  \end{split}
\end{equation}
When sending $\varepsilon \rightarrow 0$ we find for all points $s$ of $ \leb{\infty}$-weak-*-continuity of $\eta(\vec u (\cdot,\sigma))$ in $(0,t)$ that
\begin{equation}\label{Daf6}
  - \int_0^1  \ent{\eta}{\vec u(x,s)}{\vec v(x,s)} 
  \d x
  -
  \int_0^s \int_0^1 
  \pdx \vec v \D^2\eta(\vec v) (\vec f(\vec u) - \vec f(\vec v)  - \D\vec f(\vec v) (\vec u - \vec v)) \d x \d t 
+ \int_0^1  \ent{\eta}{\vec u_0}{\vec v_0}  \d x
\geq 0.
\end{equation}
Upon using \eqref{eq:consts} we infer that for almost all $s \in (0,t)$
\begin{equation}\label{Daf7}
{\Cetad} \Norm{ \vec u(\cdot,s) - \vec v(\cdot,s)}_{\leb{2}(I)}^2
\leq  \Cetau \Norm{ \vec u_0 - \vec v_0}_{\leb{2}(I)}^2
+ \Cfu \Cetau  \int_0^s \norm{\vec v(\cdot,\sigma)}_{W^{1,\infty}(I)}  \Norm{ \vec u(\cdot,\sigma) - \vec v(\cdot,\sigma)}_{\leb{2}(I)}^2 \d t 
 .
\end{equation}
This equation, in fact, holds for all $s \in (0,t)$ as $\vec u$ is weakly lower semicontinuous.
Since $\vec v$ is Lipschitz continuous, applying Gronwall's Lemma completes the proof.
\end{proof}

%

\section{The semi-discrete scheme}
\label{sec:scheme}

We will discretise \eqref{eq:cl} in space using consistent dG finite
element methods.  Let $I:=[0,1]$ be the unit interval and choose $0 =
x_0 < x_1 < \dots < x_N = 1.$ We denote $I_n=[x_n,x_{n+1}]$ to be the
$n$--th subinterval and let $h_n:= x_{n+1}-x_n$ be its size.  Let
$\poly p(I)$ be the space of polynomials of degree less than or equal to $p$ {on $I$}, then
we denote
\begin{equation}
  \fes_p 
  :=
  \ensemble{\vec g : I \to \rR^d }
  { g_i \vert _{I_n} \in \poly p{(I_n)} \text{ for } i=1,\dots,d, \ n =0,\dots, N-1},
\end{equation}
where $\vec g = \Transpose{\qp{g_1,\dots,g_d}}$, to be the usual space
of piecewise $p$--th degree polynomials for vector valued functions
over $I$. In addition we define jump and average operators such that
\begin{equation}
  \begin{split}
    \jump{\vec g}_n
    &:= 
    \vec g(x_n^-) - \vec g(x_n^+)
    := 
    \lim_{s \searrow 0} \vec g(x_n-s) - \lim_{s \searrow 0} \vec g(x_n+s),
    \\
    \avg{\vec g}_n
    &:
    = 
    \frac{1}{2} \qp{\vec g(x_n^-) +\vec g(x_n^+)}
    :=
    \frac{1}{2} \qp{\lim_{s \searrow 0} \vec g(x_n-s) + \lim_{s \searrow 0} \vec g(x_n+s)}.
  \end{split}
\end{equation}
We will examine the following class of semi-discrete numerical schemes where
$\vec u_h \in C^1([0,T),\rV_p)$ is determined such that
\begin{equation}
  \label{eq:sch2}
  \begin{split}
    0
    &=
    \sum_{n=0}^{N-1} 
    \int_{I_n}
    \qp{
      \pdt \vec u_h \cdot \vec \phi
    + 
    \pdx \vec f(\vec u_h) 
    \cdot
    \vec \phi 
    }\d x
    \\
    &\qquad
    +
    \sum_{n=0}^{N-1}
    \qp{
    \vec F(\vec u_h(x_n^-),\vec u_h(x_n^+)) 
    \cdot
    \jump{\vec \phi}_n 
    -
    \jump{\vec f(\vec u_h) \cdot \vec \phi}_n
    }
    \Foreach \vec \phi \in \rV_p.
  \end{split}
\end{equation}
In the sequel we will assume that \eqref{eq:sch2} has a solution and
in particular that $\vec u_h$ takes values in $U$. We also set
\begin{equation}\label{bc}
  \jump{\vec u_h}_0 
  :=
  \vec u_h(x_N^-) - \vec u_h(x_0^+); 
  \qquad
  \avg{\vec u_h}_0 
  := 
  \frac{\vec u_h(x_0^+)+\vec u_h(x_N^-)}{2}
\end{equation}
to account for the periodic boundary conditions. Here $\vec F: U^2
\subset \rR^{2d} \rightarrow \rR^d$ is a numerical flux function. We
restrict our attention to a certain class of numerical flux
functions. We impose that there exists a function
\begin{equation}
  \label{nfluxes}
  \vec w : U \times U \rightarrow U
  \text{  such that  }
  \vec F(\vec u,\vec v)
  =
  \vec f(\vec w(\vec u,\vec v))
\end{equation}
and that there exists a constant $L>0$ such that $\vec w$ satisfies
\begin{equation}
  \label{eq:lipschitz fluxes}
  \norm{\vec w(\vec u,\vec v) - \vec u}
  \leq 
  L
  \norm{\vec u -  \vec v},
  \qquad
  \norm{\vec w(\vec u,\vec v) - \vec v}
  \leq L \norm{\vec u - \vec v} {\Foreach \vec u, \vec v \in U}.
\end{equation}

\begin{Rem}[restriction of fluxes]\label{rem:nf1}
  The reason for the restriction on the choice of fluxes will be made
  aparant in the sequel.  Our assumptions are met obviously by upwind
  as well as central fluxes for any system under consideration. This
  is also true for Godunov schemes employing exact Riemann solvers.
  For approximate Riemann solvers there are two classes
  \cite[Sec. 12.3]{LeV02}. Our assumption is generally satisfied for
  the class in which the numerical flux is computed by evaluating the
  exact flux on some intermediate state {extracted} from an
  approximate Riemann solution.  For the second class, which
  encompasses e.g. the Roe scheme, the situation is more involved.

  Let us look at some numerical fluxes in special cases: In case of
  inviscid Burgers equation, \ie $f(u) = \tfrac{u^2}{2}$, our
  condition is not satisfied for the local and global Lax--Friedrichs
  scheme. For the local Lax--Friedrichs scheme the numerical flux
  reads
  \begin{equation}
    F(u,v)= \frac{1}{2} (u^2 + v^2) + \max(|u|,|v|)(u-v)
  \end{equation}
  which is negative for $u=0$ and $v > 0.$ Therefore there can be no
  $w \in U$ satisfying $f(w)=F(0,v).$ The argument for the global
  Lax--Friedrichs scheme is analogous.
  
  For the inviscid Burger's equation both the Roe and the
  Engquist-Osher flux satisfy our condition, with
  \begin{equation}
    w_{\text{EO}}( a, b)=\sqrt{\frac{1}{2} a^2 (1+\operatorname{sgn}(a)) +\frac{1}{2} b^2 (1-\operatorname{sgn}(b)) 
    }
  \end{equation}
  and 
  \begin{equation}
    w_{\text{Roe}}( a, b)= \sqrt{\frac{1}{2} a^2 (1+\operatorname{sgn}(a+b)) +\frac{1}{2} b^2 (1-\operatorname{sgn}(a+b)) .
    }
  \end{equation}
  The situation is far more complicated for nonlinear systems.  In
  fact, for the $p$-system which is given by
  \begin{align*}
    \pd t u - \pd x v &=0\\
    \pd t v - \pd x p(u) &=0
  \end{align*}
  for some function $p$ with $p'>0,$ the question whether the Roe scheme fits into our framework hinges on whether $p$ is surjective.
\end{Rem}

\section{Reconstruction and projection operators}
\label{sec:rec}

To analyse the scheme \eqref{eq:sch2} we introduce reconstructions
which we denote by $\hatu$ and $\hatf$. For brevity we will ommit the
time dependency of all quantities in this section.

\begin{Defn}[reconstruction of $\vec u_h$]
  \label{def:hatu}
  The reconstruction $\hatu$ is the unique element of $\rV_{p+1}$ such
  that
  \begin{equation}
    \label{hatuuc}
    \begin{split}
      \sum_{n=0}^{N-1} 
      \int_{I_n} 
      \hatu \cdot \vec \phi \d x 
      =
      \sum_{n=0}^{N-1}
      \int_{I_n}
      {\vec u}_h \cdot \vec \phi \d x \Foreach \vec \phi \in \fes_{p-1}
    \end{split}
  \end{equation}
  {and}
  \begin{gather}
    \label{eq:hatu-bcs1}
    \hatu (x_n^+)
    =
    \vec w(\vec u_h(x_n^-),\vec u_h(x_n^+)) \AND
    \\
    \label{eq:hatu-bcs2}
    \hatu (x_{n+1}^-)
    =
    \vec w(\vec u_h(x_{n+1}^-),\vec u_h(x_{n+1}^+))
    \Foreach n\in [0, N-1].
  \end{gather}
{recalling that $\vec u_h(x_0^-):=\vec u_h(x_N^-),$ and $\vec
  u_h(x_N^+):=\vec u_h(x_0^+)$.}
\end{Defn}


\begin{Defn}[reconstruction of $\vec f(\vec u_h)$]
\label{def:hatf}
  The reconstruction $\hatf$ is the unique element of $\rV_{p+1}$ such
  that
  \begin{multline}
    \label{hatfuc}
      \sum_{n=0}^{N-1} \int_{I_n} 
      \pd x \hatf 
      \cdot
      \vec \phi \d x
      =
      \sum_{n=0}^{N-1} \int_{I_n} 
      \pd{x}{\vec f(\vec u_h)} 
      \cdot
      \vec \phi \d x
      \\+
      \sum_{n=0}^{N-1} \qp{
      {\vec f}(\vec w(\vec u_h(x_n^-),\vec u_h(x_n^+)))
      \cdot 
      \jump{\vec \phi}_n 
      -
      \jump{{\vec f}({\vec u}_h) 
        \cdot 
        \vec \phi}_n} \Foreach \vec \phi\in\fes_p
 \end{multline}
coupled with the skeletal ``boundary'' conditions that
\begin{equation}
  \label{eq:hatf-bcs}
  \hatf (x_n^+) 
  =
  {\vec f}(\vec w(\vec u_h(x_n^-),\vec u_h(x_n^+)))  \Foreach n\in \qb{0, N-1}.
\end{equation}
\end{Defn}

\begin{Lem}[continuity and orthogonality]
  The reconstructions $\hatu$ and $\hatf$ given in Definitions
  \ref{def:hatu} and \ref{def:hatf} respectively are continuous and
  $\hatf$ satisfies the orthogonality property
  \begin{equation}\label{ortho}
    \sum_{n=0}^{N-1}
    \int_{I_n} 
    \qp{\hatf - \vec f(\vec u_h)}
    \cdot 
    \vec \phi \d x =0   
    \Foreach \vec \phi \in \fes_{p-1}.
  \end{equation}
\end{Lem}
\begin{proof}
  The continuity of $\hatu$ follows from
  (\ref{eq:hatu-bcs1})--(\ref{eq:hatu-bcs2}).  To prove the continuity
  of $\hatf$ we choose $\vec \phi$ as the $i$-th unit vector on $I_n$
  and zero elsewhere. Then, upon letting $\hatf =
  \Transpose{\qp{\shatf_1, \dots, \shatf_d}}$ and $\vec f =
  \Transpose{\qp{f_1,\dots,f_d}}$ we obtain from (\ref{hatfuc})
  \begin{equation}
    \begin{split}
      \shatf_i (x_{n+1}^-) 
      -
      \shatf_i (x_n^+)
      &= 
      f_i (\vec u_h (x_{n+1}^- ))
      -
      f_i (\vec u_h (x_{n}^+ ))
      -
      f_i (\vec w(\vec u_h(x_n^-),\vec u_h(x_n^+)))
      \\
      &\qquad 
      +
      f_i( \vec w(\vec u_h(x_{n+1}^-),\vec u_h(x_{n+1}^+)))
      +
      f_i (\vec u_h (x_{n}^+ )) - f_i (\vec u_h (x_{n+1}^- )).
    \end{split}
  \end{equation}
  This implies 
  \begin{equation}
    \label{rightlim} 
    \shatf_i (x_{n+1}^- ) 
    =
    f_i
    (\vec w(\vec u_h(x_{n+1}^-),\vec u_h(x_{n+1}^+)))
  \end{equation}
  due to (\ref{eq:hatf-bcs}). This shows the continuity of $\hatf$.
  Using integration by parts in \eqref{hatfuc} we have that the
  boundary terms cancel due to our choice of $\hatf (x_n^+)$ and
  \eqref{rightlim}. Hence, we find
  \begin{equation}
    \sum_{n=0}^{N-1}
    \int_{I_n}
    \hatf \cdot \pd{x}{\vec \phi} \d x  
    =
    \sum_{n=0}^{N-1}
    \int_{I_n} {\vec f}({\vec u}_h)\cdot \pd{x}{\vec \phi} \d x \Foreach {\vec \phi} \in \rV_{p}
  \end{equation}
  concluding the proof.
\end{proof}

\begin{Defn}[$\leb{2}$ projection]
  \label{def:proj}
  We define $\hP_p : \qb{\leb{2}(I)}^d \to \fes_p$ to be the $L_2$
  orthogonal projection to $\fes_p$, that is,
  \begin{equation}
    \int_I \vec \psi \cdot \vec \phi \d x 
    =
    \int_I \hP_p(\vec \psi) \cdot \vec \phi \d x  \Foreach\vec\phi \in \rV_p.
  \end{equation}
  If $\psi\in\sob{p+1}{\infty}(I)$ the operator is well known \cite[c.f.]{Cia78} to satisfy the
  following estimate in $\leb{\infty}$:
  \begin{equation}
    \label{eq:l2-proj-error}
    \Norm{\psi - \hP_p \psi}_{\leb{\infty}(I_n)} 
    \leq
    C_p h_n^{p+1} \norm{\psi}_{\sob{p+1}{\infty}} \Foreach n=0,\dots,N-1.
  \end{equation}
\end{Defn}

\begin{Rem}[restriction of fluxes revisited]
  The assumption on the numerical flux functions \eqref{nfluxes} is
  posed such that we can choose our reconstructions $\hatu, \hatf$
  such that $\hatf(x_n)=\vec f(\hatu(x_n))$ for all $n.$ This is
  needed for the proof of Lemma \ref{lem:resest} and it will be
  elaborated upon in Remark \ref{rem:nf2}.
\end{Rem}

\section{Error estimates}
\label{sec:err}
In this section we make use of the reconstruction operators from
\S\ref{sec:rec} to construct a posteriori bounds for the generic
numerical scheme (\ref{eq:sch2}). Using these reconstructions we can
rewrite our scheme as
\begin{equation}
  \label{eq:schuc2}
  0
  =
  \sum_{n=0}^{N-1} 
  \int_{I_n} 
  \pdt {\vec u}_h
  \cdot
  \vec \phi \d x 
  +
  \sum_{n=0}^{N-1}
  \int_{I_n} \pdx \hatf \cdot \vec \phi \d x \Foreach \vec \phi \in \rV_p.
\end{equation}
Since we have that $\pdt {\vec u}_h$ and $\pdx\hatf$ are piecewise
polynomials of degree $p$ we may write (\ref{eq:schuc2}) as a
pointwise equation
\begin{equation}
  \label{def:residual}
  \pdt \hatu 
  +
  \pdx {\vec f}(\hatu) 
  =
  \pdx {\vec f}(\hatu)
  -
  \pdx \hatf 
  +
  \pdt\hatu
  -
  \pdt \vec u_h
  =:
  \vec R.
\end{equation}
Using the relative entropy technique we obtain the following
preliminary error estimate:
\begin{Lem}[error bound for the reconstruction]
  \label{lem:relenuc}
  Let  $\vec u$ be the entropy solution of  \eqref{eq:cl},\eqref{eq:ibc}  then the difference between ${\vec u}$ and the reconstruction
  $\hatu$ satisfies
\begin{multline}
\Cetad \Norm{ \vec u(\cdot,s) - \hatu(\cdot,s)}_{\leb{2}(I)}^2
\leq  \Cetau\Norm{ \vec u_0 - \hatu_0}_{\leb{2}(I)}^2
\\+ (\Cfu \Cetau  \Norm{\hatu}_{W^{1,\infty}} + \Cetau^2)\int_0^s \Norm{ \vec u(\cdot,\sigma) - \hatu(\cdot,\sigma)}_{\leb{2}(I)}^2 \d {\sigma} 
+ \Norm{\vec R}_{\leb{2}(I \times (0,s))}^2
\end{multline}
for every $s \in (0,\infty),$ provided $\hatu$ takes values in $\mathfrak{O}.$
\end{Lem}

\begin{proof}
Since $\hatu$ is Lipschitz continuous, we 
multiply \eqref{def:residual} by $\D\eta(\hatu) $ and find for any Lipschitz continuous, non negative test function $\phi$  
\begin{equation}
 \int_0^\infty \int_0^1 \pdt \phi ( \eta(\vec u) - \eta(\hatu) ) + \pdx \phi ( q(\vec u) - q(\hatu) ) - \phi \D\eta(\hatu) \vec R \d x \d t
+ \int_0^1\phi(\cdot,0) \big( \eta(\vec u_0) - \eta(\hatu_0) \big)\d x
\geq 0.
\end{equation}
Using the definition of relative entropy and relative entropy flux, we may reformulate this as
\begin{equation}\label{Daf1b}
  \begin{split}
 \int_0^\infty \int_0^1 \pdt \phi ( \ent{\eta}{\vec u}{\hatu}  + \D\eta(\hatu) &(\vec u - \hatu) ) +
 \pdx \phi ( \ent{q}{\vec u}{\hatu} + \D\eta(\hatu) (\vec f(\vec u) - \vec f(\hatu)) ) \d x \d t
 \\
 &\qquad 
 - \int_0^\infty \int_0^1  \phi \D\eta(\hatu) \vec R \d x \d t 
 + \int_0^1\phi(\cdot,0)\big(  \eta(\vec u_0) - \eta(\hatu_0) \big) \d x
    \geq 0.
  \end{split}
\end{equation}
Using the Lipschitz continuous test function $ \vec \phi = \phi \D\eta(\hatu)$ in  \eqref{eq:wsol} and \eqref{def:residual} we obtain
\begin{equation}\label{Daf2b}
  \begin{split}
 \int_0^\infty \int_0^1 \pdt( \phi D\eta(\hatu)) (\vec u - \hatu)  +
    \pdx (\phi  D\eta(\hatu)) (\vec f(\vec u) - \vec f(\hatu)) &-\phi \D\eta(\hatu) \vec R  \d x \d t\\
& + \int_0^1\phi(\cdot,0) \D\eta(\hatu(\cdot, 0))(  \vec u_0 - \hatu_0 )\d x
= 0.
  \end{split}
\end{equation}
We use the product rule in \eqref{Daf2b} and combine it with \eqref{Daf1b} to obtain
\begin{equation}\label{Daf3b}
  \begin{split}
    &\int_0^\infty \int_0^1 \pdt \phi \ent{\eta}{\vec u}{\hatu}
    +
    \pdx \phi  \ent{q}{\vec u}{\hatu}   \d x \d t 
    \\
    &\qquad -
    \int_0^\infty \int_0^1 \phi( \pdt \hatu \D^2\eta(\hatu) (\vec u - \hatu) + \pdx \hatu \D^2\eta(\hatu) (\vec f(\vec u) - \vec f(\hatu))  ) \d x \d t
    \\
    &\qquad + \int_0^1\phi(\cdot,0)  \ent{\eta}{\vec u_0}{\hatu_0}  \d x
    \geq 0.
  \end{split}
\end{equation}
Using the fact that $\pdt \hatu= - \D\vec f(\hatu)\pdx {\hatu} + \vec R$ and \eqref{commute} we find
\begin{equation}\label{Daf4b}
  \begin{split}
    &\int_0^\infty \int_0^1 \pdt \phi \ent{\eta}{\vec u}{\hatu}   +
    \pdx \phi  \ent{q}{\vec u}{\hatu}   \d x \d t  
    \\
    &\qquad -
    \int_0^\infty \int_0^1 \phi( \pdx \hatu D^2\eta(\hatu) (\vec f(\vec u) - \vec f(\hatu)  - D\vec f(\hatu) (\vec u - \hatu) )) 
    \\
    &\qquad 
    -  \int_0^\infty \int_0^1 \phi \Transpose{(\vec u - \hatu)} D^2\eta(\hatu) \vec R \d x \d t+ \int_0^1\phi(\cdot,0)  \ent{\eta}{\vec u_0}{\hatu_0}  \d x
    \geq 0.
  \end{split}
\end{equation}
  Now we fix $t>0$, then for every $0<s<t$ and $\varepsilon >0$ we consider the test function $\phi(x,\sigma)$ given in (\ref{eq:test-function-phi}).
In this case we infer from \eqref{Daf4}
\begin{multline}\label{Daf5b}
  -\frac{1}{\varepsilon}
  \int_s^{s+\varepsilon} \int_0^1 \ent{\eta}{\vec u}{\hatu}
  \d x \d t 
  -
  \int_0^\infty \int_0^1 
  \phi( \pdx \hatu \D^2\eta(\hatu) 
  (\vec f(\vec u) - \vec f(\hatu)  - \D\vec f(\hatu) (\vec u - \hatu)) )
  \d x \d t 
  \\
  -
  \int_0^\infty \int_0^1 \phi \Transpose{(\vec u - \hatu)} 
  \D^2\eta(\hatu) \vec R \d x \d t
  + \int_0^1 \ent{\eta}{\vec u_0}{\hatu_0}  \d x
  \geq 0.
\end{multline}
When sending $\varepsilon \rightarrow 0$ we find for all points $s$ of $ \leb{\infty} $-weak-*-continuity of $\eta(\vec u (\cdot,\sigma))$ in $(0,t)$ that
\begin{multline}\label{Daf6b}
  - \int_0^1  \eta(\vec u(x,s)| \hatu(x,s))    \d x 
  -
  \int_0^s \int_0^1 
  \pdx \hatu \D^2\eta(\hatu) 
  (\vec f(\vec u) - \vec f(\hatu)  - \D\vec f(\hatu) (\vec u - \hatu))
  \d x \d t \\
  -  
  \int_0^s \int_0^1  \Transpose{(\vec u - \hatu)} 
  \D^2\eta(\hatu) \vec R \d x \d t + \int_0^1  
  \ent{\eta}{\vec u_0}{\hatu_0}  \d x
  \geq 0.
\end{multline}
Upon using \eqref{eq:consts} and the convexity of $\mathfrak{O}$ we infer that for almost all $s \in (0,t)$
\begin{multline}\label{Daf7b}
  \Cetad \Norm{ \vec u(\cdot,s) - \hatu(\cdot,s)}_{\leb{2}(I)}^2
  \leq  \Cetau\Norm{ \vec u_0 - \hatu_0}_{\leb{2}(I)}^2
  \\
  + 
  (\Cfu \Cetau  \Norm{\hatu}_{W^{1,\infty}} + \Cetau^2)
  \int_0^s \Norm{ \vec u(\cdot,\sigma) - \hatu(\cdot,\sigma)}_{\leb{2}(I)}^2 
  \d {\sigma}
  +  \Norm{\vec R}_{\leb{2}(I \times (0,s))}^2.
\end{multline}
This equation, in fact, holds for all $s \in (0,t)$ as $\vec u$ is weakly lower semicontinuous.
\end{proof}

\begin{Rem}[values of $\hatu$]
 Note that the condition that $\hatu$ takes values in $\mathfrak{O}$
 can be verified in an a posteriori fashion, as $\hatu$ can be explicitly computed.
\end{Rem}

{Let us note that $\vec R$ can be explicitly computed locally in every
  cell using only information from that cell and traces from the
  adjacent cells.  Still we would like to estimate $\Norm{\vec
    R}_{\leb{2}}^2$ by quantities only involving $\vec u_h$.  There
  are two reasons for doing this: Firstly we expect the new bound to
  be computationally cheaper.  Secondly, we will use this new form to
  argue why we expect our estimator to be of optimal order.}  {To
  bound $\Norm{\vec R}_{\leb{2}}^2$ by $\vec u_h$ only, it is
  imperative to have precise information on $\vec u_h - \hatu$.}  To
this end we introduce the Legendre polynomials together with some of
their properties.

\begin{Pro}[Legendre polynomials]
  \label{pro:Legendre}
  Let $l_k$ denote the $k$-th Legendre polynomial on $(-1,1)$, and
  $l_k^n$ its transformation to the interval $I_n$, \ie
  \begin{equation}
    l^n_k(x)
    =
    l_k\qp{2\qp{\frac{x-x_n}{h_n}} -1}
    .
  \end{equation}
  Let $\alpha_k := \pdx l_{k}(1)$. Then $l_k^n$ has the following properties
  \begin{gather}
    \label{Legendre1}
    (-1)^k l^n_k(x_n)= l^n_k(x_{n+1})=1,
    \\
    \label{Legendre2} 
    (-1)^{k+1} h_n \pdx l^n_{k}(x_n) = h_n \pdx l^n_{k}(x_{n+1})=2 \alpha_k,
    \\
    \label{Legendre3} 
    \int_{I_n} l^n_{j}(x)l^n_k(x) \d x 
     =
    \frac{{2}h_n}{2k+1}\delta_{kj} \leq h_n,\\
    \label{Legendre4}
     | l^n_k (x)|  \leq 1 \ \forall x \in I_n.
  \end{gather}
\end{Pro}
\begin{Lem}\label{lem:hatu}
  The reconstruction $\hatu$ given by Definition \ref{def:hatu}
  satisfies the following representation {for all $x \in I_n$}
  \begin{equation}
    \label{haturep}
    \begin{split}
      \qp{\hatu - \vec u_h}(x)
      &=
      \frac{1}{2}
      \bigg(
        (-1)^p
        \qp{
          \vec w(\vec u_h(x_n^-),\vec u_h(x_n^+))
          -
          \vec u_h(x_n^+)} 
        \\
        &\qquad\qquad\qquad\qquad
        +
        \vec w(\vec u_h(x_{n+1}^-),\vec u_h(x_{n+1}^+))
        -
        \vec u_h(x_{n+1}^-)
        \bigg)
        l^n_p(x)
        \\
        &\qquad+
      \frac{1}{2}  
      \bigg(
        (-1)^{p+1}
        \qp{
          \vec w(\vec u_h(x_n^-),\vec u_h(x_n^+))
          -
          \vec u_h(x_n^+)}
        \\
        &\qquad\qquad\qquad\qquad
        +
        \vec w(\vec u_h(x_{n+1}^-),\vec u_h(x_{n+1}^+))
        -
        \vec u_h(x_{n+1}^-)
        \bigg)
      l^n_{p+1}(x)
    \end{split}
  \end{equation}
  where $l_p^n \AND l^n_{p+1}$ are the rescaled Legendre polynomials from Proposition \ref{pro:Legendre}.
  Therefore,
  \begin{equation}
    \label{hatudiff}
    \Norm{\hatu - \vec u_h}_{\leb{2}(I_n)}^2
    \leq L^2 
    h_n 
    \qp{\big|\jump{\vec u_h}_{n}\big|^2  + \norm{\jump{\vec u_h}_{n+1}}^2 } 
  \end{equation}
  and
  \begin{equation}
    \label{hatuderv}
    \Norm{\partial_{x}^k{\hatu}}_{\leb{\infty}(I_n)} 
    \leq 
    \Norm{\partial_x^k \vec u_h}_{\leb{\infty}(I_n)}
    +
    L \frac{1}{h_n^k}
    b_k
    \qp{\big|\jump{\vec u_h}_n\big| 
      +
      \norm{\jump{\vec u_h}_{n+1}}}
  \end{equation}
  where $b_k:=  \abs{l_p}_{{k,\infty}}+\abs{l_{p+1}}_{{k,\infty}}.$
\end{Lem}

\begin{proof}
  Letting $\hatu = \Transpose{\qp{\shatu_1,\dots, \shatu_d}}$ and
  $\vec u_h = \Transpose{\qp{(u_h)_1, \dots, (u_h)_d}}$ and writing
  $\shatu_i\vert_{I_n}$ and $(u_h)_i\vert_{I_n}$ as linear
  combinations of Legendre polynomials we see that (\ref{hatuuc})
  implies 
  \begin{equation}
    (\shatu_i -  (u_h)_i)( x) = \alpha l_p^n( x) + \beta l_{p+1}^n( x) \Foreach  x \in I_n
  \end{equation}
  for real numbers $\alpha, \beta$ depending on $i$ and $n$. Using
  \eqref{Legendre1} and the boundary conditions on $\hatu$
  (\ref{eq:hatu-bcs1})--(\ref{eq:hatu-bcs2}) we obtain
  \begin{gather}
    \alpha (-1)^p -\beta (-1)^p 
    =
    \shatu_i(x_n^+) -  (u_h)_i (x_n^+) 
    =
    w_i(\vec u_h(x_n^-),\vec u_h(x_n^+)) - (u_h)_i(x_n^+)
  \end{gather}
  and
  \begin{gather}
    \alpha + \beta
    =
    \widehat u_i(x_{n+1}^-) 
    -
    (u_h)_i (x_{n+1}^-)
    =
    w_i(\vec u_h(x_{n+1}^-),\vec u_h(x_{n+1}^+))
    -
    (u_h)_i(x_{n+1}^-).
  \end{gather}
  Since
  \begin{equation}
    \begin{bmatrix}
      (-1)^p & (-1)^{p+1} \\ 1 & 1 
    \end{bmatrix}^{-1}
    =
    \frac{1}{2}
    \begin{bmatrix}
      (-1)^p & 1 \\ (-1)^{p+1} & 1 
    \end{bmatrix}
  \end{equation}
we obtain \eqref{haturep}.
Equations \eqref{hatudiff} and \eqref{hatuderv} are immediate consequences of \eqref{haturep} upon using \eqref{Legendre1}--\eqref{Legendre4}.
\end{proof}

\begin{The}\label{The:re}
  Let $\vec f \in {\cont{2}}(U,\rR^d)$  satisfy (\ref{eq:ecl})
  and let $\vec u$ be an entropy solution  of
  \eqref{eq:cl} with periodic boundary conditions. Let $\hatu$ take values in $\mathfrak{O}$, then for $0 \leq t
  \leq T$ the error between the numerical solution $\vec u_h$ and $\vec u$
  satisfies
  \begin{multline}\label{eq:there}
      \Norm{\vec u(\cdot,t) - \vec u_h(\cdot, t)}_{\leb{2}(I)}^2
      \leq 2L^2 \sum_n 
                    h_n 
                      \qp{\big|\jump{\vec u_h}_{n}\big|^2  + \norm{\jump{\vec u_h}_{n+1}}^2 } 
           \\
       +
      2\Cetad^{-1} \Big( \Norm{\vec R}_{\leb{2}(I \times (0,t))}^2 + \Cetau \Norm{\vec u_0 - \hatu_0}_{\leb{2}(I)}^2 \Big)        
            \exp\qp{
                    \int_0^t
                        \frac{\Cetau \Cfu \Norm{\pdx \hatu(\cdot,s)}_{\leb{\infty}(I)} + \Cetau^2}{\Cetad}  
                      \d s} .    
  \end{multline}
\end{The}

\begin{proof}
 Combining Lemma \ref{lem:gronwall} and Lemma \ref{lem:relenuc} we obtain
  \begin{multline}\label{eq:e1}
      \Norm{\vec u(\cdot,t) - \hatu(\cdot, t)}_{\leb{2}(I)}^2
      \leq 
      \Cetad^{-1} \Big( \Norm{\vec R}_{\leb{2}(I \times (0,t))}^2 +\Cetau \Norm{\vec u_0 - \hatu_0}_{\leb{2}(I)}^2 \Big) \\   \times      
            \exp\qp{
                    \int_0^t
                        \frac{\Cetau \Cfu \Norm{\pdx \hatu(\cdot,s)}_{\leb{\infty}(I)} + \Cetau^2}{\Cetad}  
                      \d s} .    
  \end{multline}
Upon using triangle inequality and \eqref{hatudiff} equation \eqref{eq:e1} implies the assertion of the Theorem.
\end{proof}

\begin{Rem}[values of $\hatu$]
  The $\leb{\infty}$ estimates based on \eqref{haturep} can be employed to verify a posteriori that $\hatu$ takes values in $\mathfrak{O}.$
\end{Rem}

\begin{Rem}[discontinuous entropy solutions]\label{rem:des}
  The estimate in Theorem \ref{The:re} does not require the entropy solution $\vec u$ to be continuous.
  However, in case $\vec u$ is discontinuous $\Norm{\pdx \hatu(\cdot,s)}_{\leb{\infty}(I)}$ is expected to behave like $\Oh(h^{-1})$.
  Therefore, the estimator in \eqref{eq:there} will (at best) be $\Oh(h^{p+1} \exp(h^{-1}))$ which diverges for $h \rightarrow 0.$
Thus, the estimator in \eqref{eq:there} is expected not to converge for $h \rightarrow 0$ if the entropy solution is discontinuous.
The same is true for the estimator derived in Theorem \ref{lem:apee}.
This is a consequence of the use of the relative entropy framework and the fact that the entropy solution does not need to be unique 
if it is not Lipschitz.
\end{Rem}

\begin{Lem}[inverse inequality {\cite[c.f.]{Cia78}}]
  \label{invineq}
  For every $k \in \rN$ there is a constant $C_{\text{inv}} > 0$ such that
  for any interval $J \subset \rR$ and any $\phi \in \poly{k}(J)$ the
  following inequality is satisfied
  \begin{equation}
    \Norm{\pdx \phi}_{L^2(J)} 
    \leq 
    \frac{C_{\text{inv}}}{|J|} 
    \Norm{\phi}_{L^2(J)}.
  \end{equation}
\end{Lem}

\begin{Lem}[a posteriori control on {$\vec R$}]
  \label{lem:resest}
  Let $\vec f \in {\cont{p+2}}(U,\rR^d)$ and satisfy (\ref{eq:consts}). It
  then holds that
  \begin{equation}
    \Norm{\vec R}_{\leb{2}(I)}^2 \leq 3 (E_1 + E_2 + E_3)
  \end{equation}
  with
  \begin{equation}
    \begin{split}
      E_1
      &:=
      \sum_{n=0}^{N-1} 
      h_n
      L^2
      \qp{
        \big|\jump{\pdt\vec u_h}_n\big|^2  + \big|\jump{\pdt\vec u_h}_{n+1}\big|^2 },
      \\
      E_2
      &:=
      \sum_{n=0}^{N-1} 
      4h_n L^2
      \qp{\big|\jump{\vec u_h}_n\big|^2  + \big|\jump{\vec u_h}_{n+1}\big|^2 }
      \qp{L
        \frac{\big|{\jump{\vec u_h}_n\big|} + \big|{\jump{\vec u_h}_{n+1}\big|}}{h_n}
        +
        \Norm{\del_x \vec u_h}_{\leb{\infty}(I_n)}}\Cfu
      \\
      &
      \qquad
      + 2h_n
      \Bigg(
      \sum_{k=0}^{p} {p+1 \choose k} 
      \Big( h_n^{p+1} 
      \Norm{\del_x^{k+1} \vec u_h}_{\leb{\infty}(I_n)} 
      +
      L h_n^{p-k} b
      \qp{\big|\jump{\vec u_h}_n\big| + \big|\jump{\vec u_h}_{n+1}\big|}
      \Big)
      \\
      & \qquad\qquad\qquad\qquad\qquad\qquad\qquad\qquad\qquad\qquad\qquad\qquad\qquad\qquad
      \times \norm{ \del_x^{p+1-k} \D\vec f(\vec u_h)}\Bigg)^2,
      \\
      E_3 
      &:=
      2C_{inv}^2 L^2
     \Cfu^2
      \norm{\vec u_h}_{W^{1,\infty}}^2
      \sum_{n=0}^{N-1} 
      h_n
      \qp{
        \big|\jump{\vec u_h}_n\big|^2 + \big|\jump{\vec u_h}_{n+1}\big|^2}   \\
      &\qquad\qquad
      +
      16C_{inv}^2 L^4 \Cfu^2
      \sum_{n=0}^{N-1}
      \frac{1}{h_n}
      \qp{ \big|{\jump{\vec u_h}_n}\big|^4 +  \big|{\jump{\vec u_h}_{n+1}\big|}^4  }
    \end{split}
  \end{equation}
  where 
  $b:= \| l_p\|_{W^{p+1,\infty}} +  \| l_{p+1}\|_{W^{p+1,\infty}}$.
\end{Lem}

\begin{proof}
  Recalling the definition of $\vec R$
  \begin{equation}
  \vec R
  :=    
  \pdt \hatu 
  +
  \pdx {\vec f}(\hatu) 
  =
  \pdx {\vec f}(\hatu)
  -
  \pdx \hatf 
  +
  \pdt\hatu
  -
  \pdt \vec u_h,
  \end{equation}
  we begin by splitting $\vec R$ into three quantities via the
  $\leb{2}$ projection of $\pdx \vec f(\hatu)$, that is,
  \begin{equation}
    \vec R
    =
    \pdt\qp{\hatu - \vec u_h} 
    +
    \qp{\pdx \vec f(\hatu) - \hP_p\qp{\vec f (\hatu)_x}}
    +
    \qp{\hP_p\qp{\vec f (\hatu)_x} - \hatf_x} =: \vec R_1 + \vec R_2 + \vec R_3,
  \end{equation}
  and bounding each of these individually.

  Forming the time derivative of \eqref{hatudiff} we immediately obtain
  \begin{equation}
    \label{190601} 
    \Norm{\vec R_1}_{\leb{2}(I_n)}^2
    =
    \Norm{\pdt(\hatu - \vec u_h)}_{\leb{2}(I_n)}^2 
    \leq
    L^2 h_n 
    \qp{ \big|\jump{\pdt \vec u_h}_n\big|^2  + \big|\jump{\pdt\vec u_h}_{n+1}\big|^2 }.
  \end{equation}
  For the term involving $\vec R_2$ we further split the term and evaluate derivatives, giving 
  \begin{equation}
    \label{eq:260600}
    \begin{split}
      \Norm{\hP_p \qp{\pdx \vec f (\hatu)} - \pdx \vec f(\hatu)}_{\leb{2}(I_n)} 
      &\leq 
      \Norm{\hP_p\qp{\D\vec f (\hatu) \pdx\hatu}
        -
        \hP_p
        \qp{
          \D\vec  f(\vec u_h )\pdx\hatu
        }
      }_{\leb{2}(I_n)} 
      \\
      &\qquad +
      \Norm{ \D\vec  f(\vec u_h ) \pdx \hatu
        -
        \D \vec f (\hatu) \pdx \hatu }_{\leb{2}(I_n)} 
      \\
      &\qquad +
      \Norm{\hP_p\qp{\D\vec f (\vec u_h) \pdx\hatu}
        -
        \D\vec f(\vec u_h)\pdx \hatu}_{\leb{2}(I_n)} 
      \\
      &\leq 
      2 \Norm{ \pdx \hatu}_{\leb{\infty}(I_n)}
      \Cfu
      \Norm{ \hatu  - \vec u_h   }_{\leb{2}(I_n)} 
      \\
      & \qquad +
      \Norm{\hP_p\qp{\D\vec f (\vec u_h) \pdx\hatu} - \D\vec f(\vec u_h) \pdx \hatu}_{\leb{2}(I_n)} 
      {}
    \end{split}
  \end{equation}
  since the $L_2$-projection is stable and satisfies $\Norm{\hP_p
    g}_{\leb{2}(\W)} \leq \Norm{g}_{\leb{2}(\W)}$ for any
  $g\in\leb{2}(\W)$. In addition from (\ref{eq:l2-proj-error}) we have
  that 
  \begin{equation}
    \label{eq:260601}
    \Norm{\hP_p\qp{\D\vec f (\vec u_h) \pdx \hatu}
      - \D\vec f(\vec u_h) \pdx\hatu}_{\leb{\infty}(I_n)}
    \leq 
    C_p h_n^{p+1} 
    \norm{\D\vec f(\vec u_h)\pdx\hatu}_{\sob{p+1}{\infty}(I_n)}.
  \end{equation}
  By the product rule we have inside $I_n$
  \begin{equation}
    \label{eq:260602}
    \begin{split}
      \del_x^{p+1}
      \qp{ \D\vec f(\vec u_h) \pdx \hatu}
      &=
      \sum_{k=0}^{p+1} {p+1 \choose k}\qp{\del_x^{k+1} \hatu } \qp{ \del_x^{p+1-k} \D\vec f(\vec u_h)}
      \\
      &=
      \sum_{k=0}^{p} {p+1 \choose k}\qp{\del_x^{k+1} \hatu } \qp{ \del_x^{p+1-k} \D\vec f(\vec u_h)}.
    \end{split}
  \end{equation}
  as $\hatu \in \rV_{p+1}.$ Using the properties of the derivatives of
  the reconstruction \eqref{hatuderv} in \eqref{eq:260602} we have that
  \begin{equation}
    \label{eq:260603}
    \begin{split}
      h_n^{p+1}
      \Norm{ \del_x^{p+1} \qp{ \D\vec f(\vec u_h)\cdot \hatu_x}}_{\leb{\infty}(I_n)}
      &\leq
       h_n^{p+1}
      \sum_{k=0}^p {p+1 \choose k} \Norm{\partial_x^{k+1} \hatu}_{\leb{\infty}(I_n)} \Norm{\partial_x^{p+1-k} \D \vec f(\vec u_h)}_{\leb{\infty}(I_n)}
      \\
      &\leq \sum_{k=0}^{p}{p+1 \choose k} 
      \bigg(
      h_n^{p+1} \Norm{\del_x^{k+1} \vec u_h}_{\leb{\infty}(I_n)} 
      \\ 
      &\qquad + Lh_n^{p-k} b_{k+1}
      \qp{\big|\jump{\vec u_h}_n\big| +\big|\jump{\vec u_h}_{n+1}\big|}\bigg)
      \Norm{ \del_x^{p+1-k}  \D \vec f(\vec u_h)}_{\leb{\infty}(I_n)}.
    \end{split}
  \end{equation}
  Inserting \eqref{eq:260603} into \eqref{eq:260601} gives
  \begin{equation}
    \label{eq:260604}
    \begin{split}
      &\Norm{\hP_p\qp{\D\vec f (\vec u_h)\pdx\hatu}
        - \D\vec f(\vec u_h)\pdx\hatu}_{\leb{\infty}(I_n)} 
      \\
      &\qquad\qquad\qquad \leq 
      C_p  \sum_{k=0}^{p} \bigg( {p+1 \choose k}
      \bigg( h_n^{p+1} \Norm{\del_x^{k+1} \vec u_h}_{\leb{\infty}(I_n)} 
        \\
        &\qquad\qquad\qquad\qquad\qquad\qquad+
        L h_n^{p-k} b_{k+1}
        \qp{\big|\jump{\vec u_h}_n\big| +\big|\jump{\vec u_h}_{n+1}\big|}\bigg)
      \Norm{ \del_x^{p+1-k} \D \vec f(\vec u_h)}_{\leb{\infty}(I_n)} \bigg).
    \end{split}
  \end{equation}
  Therefore, we can infer from \eqref{eq:260600} that
  \begin{equation}
    \label{eq:260605}
    \begin{split}
      \Norm{\vec R_2}_{\leb{2}(I_n)}^2 
      &\leq
      8 \Cfu L^2 h_n 
      \qp{\big|\jump{\vec u_h}_n\big|^2  + \big|\jump{\vec u_h}_{n+1}\big|^2 }
      \Norm{\pd x \hatu}_{\leb{\infty}(I_n)} 
      \\
      &\qquad + 
      2 C_p^2 h_n \Bigg( \sum_{k=0}^{p} {p+1 \choose k} \Big( h_n^{p+1} \Norm{\del_x^{k+1} \vec u_h}_{\leb{\infty}(I_n)} 
      \\
      &\qquad\qquad
      +
      h_n^{p-k} b_{k+1}
      \qp{\big|\jump{\vec u_h}_n\big| + \big|\jump{\vec u_h}_{n+1}\big|} \Norm{ \del_x^{p+1-k} \D\vec f(\vec u_h)}_{\leb{\infty}(I_n)}\Big)\Bigg)^2.
    \end{split}
  \end{equation}
Using the fact that
\begin{equation}
  \Norm{\del_x \hatu}_{\leb{\infty}(I_n)}
  \leq
  L\frac{\big|\jump{\vec u_h}_n\big| + \big|\jump{\vec u_h}_{n+1}\big|}{h_n} + \Norm{\pd x {\vec u_h}}_{\leb{\infty}(I_n)}
\end{equation}
equation \eqref{eq:260605} implies the desired estimate for
$\Norm{\vec R_2}_{\leb{2}(I)}^2$.

To conclude we will estimate the term containing $\vec R_3$. Note that $\vec R_3
\in \rV_p.$ Using the definitions of $\hatu$ and $\hatf$ as well as
integration by parts we find 
\begin{equation}
  \label{eq:r3}
  \begin{split}
    \Norm{\vec R_3}_{\leb{2}(I)}^2 
    &=
    \sum_{n=0}^{N-1} \int_{I_n} \norm{\vec R_3}^2 \d x
    =
    \sum_{n=0}^{N-1} \int_{I_n} \qp{\hP_p \qp{\pd x{ \vec f(\hatu)}} - \pd x \hatf} \cdot \vec R_3 \d x
    \\
    &= 
    \sum_{n=0}^{N-1} \int_{I_n} \qp{\pd x {\vec f (\hatu)} - \pd x \hatf } \cdot \vec R_3\d x\\
    &
    =
    \sum_{n=0}^{N-1} \int_{I_n} \qp{\pd x {\vec f (\hatu)} - \pdx {\vec f(\vec u_h)} } \cdot \vec R_3\d x
    \\
    & \qquad\qquad   - \sum_{n=0}^{N-1}\qp{ {\vec f}(\vec w(\vec u_h(x_n^-),\vec u_h(x_n^+)))\cdot \jump{\vec R_3}_n+ \jump{{\vec f}({\vec u}_h)\cdot \vec R_3}_n}.
\end{split}
\end{equation}
Now upon integrating by parts, we see that
\begin{equation}
  \begin{split}
    \Norm{\vec R_3}_{\leb{2}(I)}^2 
    &=
    -\sum_{n=0}^{N-1} \int_{I_n} \qp{\vec f (\hatu) - \vec f(\vec u_h) } \cdot \pd x {\vec R_3}\d x.
\end{split}
\end{equation}

Using the orthogonality property (\ref{hatuuc}) taking $\phi = \D \vec
f(\hP_0 \vec u_h)$ we have that
\begin{equation}\label{123}
  \begin{split}
    \Norm{\vec R_3}_{\leb{2}(I)}^2
    &
    \leq 
    \sum_{n=0}^{N-1} 
    \int_{I_n}
    \bigg[ 
    \qp{\D \vec f(\hP_0 \vec u_h) - \D \vec f(\vec u_h)} \qp{\hatu - \vec u_h}
    \\
    &\qquad  
    +
    \sum_{|\vec \beta|=2}
    \qp{\frac{2}{\vec \beta!}
      \int_0^1 (1-t) \D^{\vec\beta} 
      \vec f ( \vec u_h + t(\hatu - \vec u_h)) \d t}
    (\hatu - \vec u_h)^{\vec \beta}\bigg] \pdx \vec R_3
    \d x
    \\
    &
    \leq C_{inv}  \Cfu
    \norm{\vec u_h}_{\sob{1}{\infty}}
    \Norm{\hatu - \vec u_h}_{\leb{2}(I)} 
    \Norm{\vec R_3}_{\leb{2}(I)}\\
    &\qquad+ C_{inv}\Cfu
    \sqrt{ \sum_{n=0}^{N-1} \frac{1}{h_n^2}   \int_{I_n} \norm{\hatu - \vec u_h}^4\d x} \Norm{\vec R_3}_{\leb{2}(I)},
 \end{split}
\end{equation}
by the inverse inequality \eqref{invineq}, where $\D^{\vec\beta}\vec
f$ is the partial derivative of $\vec f$ specified by the multiindex
$\vec \beta$.  Note that $ \norm{\vec u_h}_{\sob{1}{\infty}} $ in \eqref{123} is to be understood as $\max_{n=1,\dots,N} \norm{\vec u_h|_{I_n}}_{\sob{1}{\infty}(I_n)}$.
 Therefore,
\begin{equation}
  \Norm{\vec R_3}_{\leb{2}(I)}
  \leq C_{inv}\Cfu
  \qp{
  \norm{\vec u_h}_{\sob{1}{\infty}}
  \Norm{\hatu - \vec u_h}_{\leb{2}(I)} 
  +
  \sqrt{ \sum_{n=0}^{N-1} \frac{1}{h_n^2}
    \int_{I_n} \norm{\hatu - \vec u_h}^4\d x
  }}.
\end{equation}
In view of the boundedness of the Legendre polynomials and \eqref{haturep} this implies
\begin{equation}
  \begin{split}
  \Norm{\vec R_3}_{\leb{2}(I)}
  &\leq C_{inv}\Cfu \bigg(\norm{\vec u_h}_{\sob{1}{\infty}}\sqrt{\sum_{n=0}^{N-1} h_n L^2\qp{ \big|{\jump{\vec u_h}_n}\big|^2 +\big|{\jump{\vec u_h}_{n+1}\big|^2}  }}
   \\
   &\qquad +
    \sqrt{ \sum_{n=0}^{N-1} \frac{1}{h_n} L^4 \qp{ \big|{\jump{\vec u_h}_n}\big|^4 +  \big|{\jump{\vec u_h}_{n+1}}\big|^4    } \d x} \bigg),
\end{split}
\end{equation}
concluding the proof.
\end{proof}

\begin{Rem}[general numerical fluxes]
  \label{rem:nf2}
  The assumption on the numerical fluxes \eqref{nfluxes} was used in the above proof
  in order to estimate $\vec R_3.$ If we used more general numerical
  fluxes we would get additional contributions in the estimate
  \eqref{eq:r3} which would not be of optimal order in general. In
  particular, it is not sufficient for the numerical fluxes to be consistent and
  monotone.
\end{Rem}

\begin{Lem}[stability of the reconstruction]
  \label{lem:stab}
  Let $\vec f \in {\cont{p+2}}(U,\rR^d)$  satisfy (\ref{eq:ecl})
  and  let  $\vec u$ be an entropy solution of
  \eqref{eq:cl} with periodic boundary conditions. Then, provided $\hatu$ takes values in $\mathfrak{O},$ for $0 \leq t
  \leq T$ the error between the reconstruction $\hatu$ and $\vec u$
  satisfies
  \begin{equation}
    \begin{split}
      \Norm{\vec u(\cdot,t) - \hatu(\cdot, t)}_{\leb{2}(I)}^2
      &\leq 
      \Cetad^{-1} E(t){}
      \exp\qp{\int_{{0}}^t\frac{\Cetau \Cfu \Norm{\pdx \hatu(\cdot,\sigma)}_{\leb{\infty}(I)}
          + \Cetau^2}{\Cetad}\d \sigma}{}
    \end{split}
  \end{equation}
with
\begin{equation}\label{def:E}
\begin{split}
  E(t)
  &
  :=
  {\Cetau}\ent{\eta}{\vec u(\cdot, 0)}{\hatu(\cdot,0))}
  \\
  &\qquad 
  +
  \int_0^t 
  3 \sum_{n=0}^{N-1} h_n 
  \Bigg[ {L}^2 \qp{\big|\jump{\pdt {\vec u_h}}_n\big|^2  + \big|\jump{\pdt {\vec u_h}}_{n+1}\big|^2 }
  \\
  &
  \qquad 
  +
  4 L^2 
  \qp{\big|\jump{\vec u_h}_n\big|^2  + \big|\jump{\vec u_h}_{n+1}\big|^2 }
  \qp{L \frac{\big|{\jump{\vec u_h}_n}\big| + \big|{\jump{\vec u_h}_{n+1}}\big|}{h_n}
    +
    \Norm{\pdx {\vec u_h}}_{\leb{\infty}(I_n)}} 
  {\Cfu}
  \\
  &
  \qquad +
  2 \Bigg( \sum_{k=0}^{p}{p+1 \choose k} \Big( h_n^{p+1} \Norm{\partial_x^{k+1} \vec u_h}_{\leb{\infty}(I_n)} 
  +
  |h_n|^{p-k} L b_k
  \qp{\big|\jump{\vec u_h}_n \big| + \big|\jump{\vec u_h}_{n+1}\big|}
  \Big)
  \\
  & \qquad \qquad \times \Norm{ \del_x^{p+1-k} \D \vec f(\vec u_h)}_{\leb{\infty}(\W)}\Bigg)^2
  \\
  & \qquad +
  2C_{inv}^2 {\Cfu^2}
  \norm{\vec u_h}_{W^{1,\infty}}^2
  L^2 \qp{ \big|{\jump{\vec u_h}_n\big|^2 +\big|\jump{\vec u_h}_{n+1}}^2\big|}
  \\
  &
  \qquad 
  +
  16C_{inv}^2 {\Cfu^2}\frac{1}{h_n^2}
  L^4\qp{ \big|\jump{\vec u_h}_n\big|^4 +  \big|\jump{\vec u_h}_{n+1}\big|^4    }
  \Bigg] \d s,
\end{split}
\end{equation}
All the quantities inside the integral on the right hand side of
\eqref{def:E} are evaluated at time $s$.
\end{Lem}
\begin{proof}
  The proof follows by combining Lemmas \ref{lem:relenuc} and
  \ref{lem:resest}.
\end{proof}

\begin{The}[a posteriori error estimate]\label{lem:apee}
  Let $\vec f \in{\cont{p+2}}(U,\rR^d)$ and $\vec u$ be the entropy
  solution of \eqref{eq:cl} with periodic boundary conditions.
  Let  $\hatu$ takes values in $\mathfrak{O}.$ Then for
  $0 \leq t \leq T$ the error between the numerical solution $\vec
  u_h$ and $\vec u$ satisfies
  \begin{equation}\label{eq:apee}
    \begin{split}
      \Norm{\vec u(\cdot,t) - \vec u_h(\cdot, t)}_{\leb{2}(I)}^2
      &\leq
      \Cetad^{-1}{E(t)} {}
      \exp\qp{\int_{{0}}^t\frac{\Cetau\Cfu \Norm{\pdx \hatu(\cdot,\sigma)}_{\leb{\infty}(I)} + \Cetau^2}{\Cetad}\d \sigma}{}
      \\
      &\qquad 
      +
      {} L^2 \sum_n h_n \qp{ \big|{\jump{\vec u_h(\cdot, t)}_n}\big|^2 + \big|{\jump{\vec u_h(\cdot, t)}_{n+1}}\big|^2}
    \end{split}
  \end{equation}
  where $E$ is defined as in Lemma \ref{lem:stab}.
\end{The}
\begin{proof}
  The proof follows from Lemma \ref{lem:hatu} and Lemma \ref{lem:stab}.
\end{proof}

\begin{Rem}[optimality of the estimator]
  Assume that the entropy solution $\vec u$ and its time derivative $\pd t {\vec u}$ are $p+1$ 
times continuously differentiable in space and 
\begin{equation}
  \Norm{\vec u - \vec u_h}_{\leb{\infty}(0,T;\leb{2}(I))}
  +
  \Norm{\pd t {\vec u} - \pd t {\vec u_h}}_{\leb{\infty}(0,T;\leb{2}(I))}
  \leq
  C h^{p+1}.
\end{equation}
In that case it is expected that $\Norm{\pdx \hatu}_{\leb{\infty}(0,T;\leb{\infty}(I))}$ is bounded uniformly in $h$ and, moreover, the arguments from 
\cite[Rem 3.6]{MN06} indicate that 
 \[ \sum_n h_n \qp{\big|\jump{\pdt {\vec u_h}}_n\big|^2  + \big|\jump{\pdt {\vec u_h}}_{n+1}\big|^2 } \leq C h^{2p+2} \text{ and }
  \sum_n h_n \qp{\big|\jump{ {\vec u_h}}_n\big|^2  + \big|\jump{ {\vec u_h}}_{n+1}\big|^2 } \leq C h^{2p+2}
 \]
where $h=\max_n h_n.$
As, in addition,
\[ \frac{1}{h_n}\qp{\big|\jump{ {\vec u_h}}_n\big|  + \big|\jump{ {\vec u_h}}_{n+1}\big| } \]
is expected to be bounded, we expect $E$ in \eqref{eq:apee} to be of order $h^{2p+2}$ and the exponential term in \eqref{eq:apee} to be  bounded uniformly  in $h.$
Therefore, we claim  that our error estimator is of optimal order, for sufficiently smooth solutions.
This is supported by numerical evidence in Section \ref{sec:num}.
\end{Rem}

\begin{Rem}
 As can be seen in \cite{Daf10} the relative entropy stability estimate in Lemma \ref{lem:ltstab}
 can be localized in the sense that there is a computable $c>0$ depending on $\mathfrak{O}$
such that for every $[a,b] \subset I$  and $t>0$  
  \begin{equation}
    \Norm{\vec u(\cdot,t) - \vec v (\cdot, t) }_{\leb{2}([a,b])}
    \leq 
    C_1 \exp( C_2 t) \Norm{ \vec u_0 - \vec v_0}_{\leb{2}([a-ct,b+ct])}.
  \end{equation}
with $C_2$ depending on $\Norm{\del_x \vec v}_{\leb{\infty}(\{ (x,s) : x \in [a -cs , b+cs]\})}.$
This, in particular, shows that the arguments presented above allow for the construction of localized a posteriori error estimates.
\end{Rem}

\newcommand{\est}{{\cE}}


\section{Numerical experiments}\label{sec:num}

In this section we study the numerical behaviour of the error
indicators and compare this behaviour with the true error on two model
problems. The coding was done in \matlab under the framework provided
by \cite{HW08}. 

\begin{Defn}[estimated order of convergence]
  \label{def:EOC}
  Given two sequences $a(i)$ and $h(i)\downto0$,
  we define estimated order of convergence
  (\EOC) to be the local slope of the $\log a(i)$ vs. $\log h(i)$
  curve, i.e.,
  \begin{equation}
    \EOC(a,h;i):=\frac{ \log(a(i+1)/a(i)) }{ \log(h(i+1)/h(i)) }.
  \end{equation}
\end{Defn}

\begin{Rem}[computed a posteriori indicator]
  We define
  \begin{equation}\label{eq:computed-apee}
    \begin{split}    
      \cE_t
      &:= 
      {\widetilde E(t)} {}
      \exp\qp{\int_{{0}}^t \Norm{\pdx \vec u_h(\cdot,\sigma)}_{\leb{\infty}(I)} 
        + 
        \frac{1}{h_n}
        \qp{\big|
            {\jump{\vec u_h(\cdot, t)}_n}\big|
            +
            \big|{\jump{\vec u_h(\cdot, t)}_{n+1}}\big|}
        \d \sigma}
      \\
      &\qquad +
      \sum_n h_n \qp{ \big|{\jump{\vec u_h(\cdot, t)}_n}\big|^2 + \big|{\jump{\vec u_h(\cdot, t)}_{n+1}}\big|^2},
    \end{split}
  \end{equation}
  where
\begin{equation}\label{def:computed-E}
  \begin{split}
    \widetilde E(t)
    &
    :=
    \ent{\eta}{\vec u(\cdot, 0)}{\hatu(\cdot,0))}
    +
    \int_0^t 
    \sum_{n=0}^{N-1} h_n 
    \Bigg[\qp{\big|\jump{\pdt {\vec u_h}}_n\big|^2  + \big|\jump{\pdt {\vec u_h}}_{n+1}\big|^2 }
      \\
      &
      \qquad \qquad 
      +
      \qp{\big|\jump{\vec u_h}_n\big|^2  + \big|\jump{\vec u_h}_{n+1}\big|^2 }
      \qp{\frac{\big|{\jump{\vec u_h}_n}\big| + \big|{\jump{\vec u_h}_{n+1}}\big|}{h_n}
        +
        \Norm{\pdx {\vec u_h}}_{\leb{\infty}(I_n)}}\Bigg].
  \end{split}
\end{equation}
Note that $\cE_t$ is equivalent to the bound in Lemma \ref{lem:apee} up to a constant in view of inverse inequalities. As such, $\cE_t$ is an \emph{aposteriori indicator}.
\end{Rem}

\begin{Defn}[effectivity index]
  \label{def:EI}
  The main tool deciding the quality of an estimator is the
  effectivity index (\EI) which is the ratio of the error and the
  estimator, \ie
  \begin{equation}
    \EI(\tn):= 
    \frac{\max_{t} \est_t}{\Norm{\vec u - \vec u_h}_{\leb{\infty}(0,T;\leb{2}(S^1))}}.
  \end{equation}
\end{Defn}

In both tests below for the temporal discretisation we choose an
explicit fourth order Runge-Kutta method. To test the asymptotic
behaviour of the estimator given in Theorem \ref{lem:apee}
we use a uniform timestep and uniform meshes that are fixed with
respect to time.  Hence for each test we have $\V n=\V 0=\fes$ and
$\tau_n=\tau(h)$ for all $n\in[1:N]$. We fix
the polynomial degree $p$ and two parameters $k,c$ and then compute a
sequence of solutions with $h=h(i)=2^{-i}$, and $\tau=ch^k$ for a
sequence of refinement levels $i=l,\dotsc,L$. 

\subsection{Test 1 : The scalar case - inviscid Burger's equation}
\label{sec:Burger}

We conduct a benchmarking experiment using the inviscid (scalar)
Burger's equation, \ie
\begin{equation}
  \label{eq:burger}
  \pdt u + \pdx \qp{\frac{u^2}{2}} = 0.
\end{equation}
Using an initial condition $u(x,0) = -\sin x$ over an interval $I =
[-\pi,\pi]$. It can be verified that, before shock formation, the
exact solution can be represented by an infinite sum of Bessel
functions, that is,
\begin{equation}
  \label{eq:test-1}
  u(x,t) = -2 \sum_{k=1}^\infty \frac{J_k(kt)}{kt} \sin{{kx}},
\end{equation}
where $J_k$ denotes the $k-th$ Bessel function. Note this is a
decaying sequence, hence we may approxmiate the solution by taking a
truncation of this series.

We discretise the problem (\ref{eq:burger}) using the dG scheme
(\ref{eq:sch2}) together with Engquist--Osher type fluxes. These fluxes
satisfy the assumptions (\ref{nfluxes})--(\ref{eq:lipschitz fluxes})
as shown in Remark \ref{rem:nf1}. Table
\ref{table:p-1} summarises the results for this test.

\begin{table}
  \caption{\label{table:p-1} In this test we computationally study the behaviour of the a posteriori indicator when the exact solution to the problem is given by (\ref{eq:test-1}). Note that this solution is only valid \emph{before} shock formation, as such $u$ is smooth.}
  \centering
  \subfloat[ A simulation with $p=1$.
  ]{
    \begin{tabular}{c|c|c|c|c|c}
$N$ & $\Norm{e_u}_{\leb{\infty}(\leb{2})}$ & EOC & $\max_t {\est_t}$ & EOC & EI \\
\hline
8 & 2.3336e-01 & 0.000 & 3.5500e-01 & 0.000 & 1.521 \\
16 & 8.6657e-02 & 1.429 & 1.3541e-01 & 1.390 & 1.563 \\
32 & 3.1863e-02 & 1.443 & 5.1422e-02 & 1.397 & 1.614 \\
64 & 1.1753e-02 & 1.439 & 1.9416e-02 & 1.405 & 1.652 \\
128 & 4.2916e-03 & 1.453 & 7.1950e-03 & 1.432 & 1.677 \\
256 & 1.5501e-03 & 1.469 & 2.6220e-03 & 1.456 & 1.692 \\
512 & 5.5526e-04 & 1.481 & 9.4403e-04 & 1.474 & 1.700 \\
1024 & 1.9779e-04 & 1.489 & 3.3723e-04 & 1.485 & 1.705 \\
2048 & 7.0216e-05 & 1.494 & 1.1990e-04 & 1.492 & 1.708 \\
4096 & 2.4879e-05 & 1.497 & 4.2518e-05 & 1.496 & 1.709 \\
\end{tabular}

  }
  \\
  \subfloat[ A simulation with $p=2$.
  ]{
    \begin{tabular}{c|c|c|c|c|c}
$N$ & $\Norm{e_u}_{\leb{\infty}(\leb{2})}$ & EOC & $\max_t {\est_t}$ & EOC & EI \\
\hline
8 & 2.2135e-02 & 0.000 & 9.5664e-02 & 0.000 & 0.000 \\
16 & 2.7472e-03 & 3.010 & 1.4455e-02 & 2.726 & 5.262 \\
32 & 3.4615e-04 & 2.988 & 2.1409e-03 & 2.755 & 6.185 \\
64 & 4.4617e-05 & 2.956 & 3.3881e-04 & 2.660 & 7.594 \\
128 & 5.6079e-06 & 2.992 & 5.1629e-05 & 2.714 & 9.207 \\
256 & 7.0465e-07 & 2.992 & 7.4392e-06 & 2.795 & 10.557 \\
512 & 8.8207e-08 & 2.998 & 1.0230e-06 & 2.862 & 11.598 \\
1024 & 1.1040e-08 & 2.998 & 1.3598e-07 & 2.911 & 12.317 \\
\end{tabular}

  }
\end{table}

\subsection{Test 2 : The system case - the $p$--system}
\label{sec:Euler}

In this case we conduct some benchmarking using the $p$--system, given
by:
\begin{equation}
  \label{eq:psystem}
  \begin{split}
    0 
    &=
    \pdt u - \pdx v
    \\
    0
    &=
    \pdt v - \pdx{\qp{p(u)}}.
  \end{split}
\end{equation}
We choose an initial condition $u(x,0) = \exp\qp{-10\norm{x}^2}$ and $v(x,0) = 0$ over an interval $I = [-5,5]$. 

We discretise (\ref{eq:psystem}) using the dG scheme (\ref{eq:sch2}) with a Roe flux (as described in Remark \ref{rem:nf1}). 
This class of fluxes satisfies the assumption on the fluxes (\ref{eq:lipschitz fluxes}) assuming $p$ is surjective. We take $p(u) = u^3 + u$.

We run the simulation on a sufficiently refined mesh and timestep to generate an accurate approximation to the solution and test the approximation rates for the method using this as a representation to the exact solution. Table \ref{table:psys-p-1} summaries the results for this test.

\begin{table}
  \caption{\label{table:psys-p-1}In this test we computationally study the behaviour of the a posteriori indicator applied to the $p$--system.}
  \centering
  \subfloat[ A simulation with $p=1$.
  ]{
    \begin{tabular}{c|c|c|c|c|c}
$N$ & $\Norm{e_u}_{\leb{\infty}(\leb{2})}$ & EOC & $\max_t {\est_t}$ & EOC & EI \\
\hline
16 & 1.5296e+00 & 0.000 & 3.2527e+00 & 0.000 & 2.126 \\
32 & 5.6355e-01 & 1.441 & 1.2362e+00 & 1.396 & 2.194 \\
64 & 2.0724e-01 & 1.443 & 4.6672e-01 & 1.405 & 2.252 \\
128 & 7.5565e-02 & 1.455 & 1.7283e-01 & 1.433 & 2.287 \\
256 & 2.7373e-02 & 1.465 & 6.3085e-02 & 1.454 & 2.305 \\
512 & 9.7873e-03 & 1.484 & 2.2669e-02 & 1.477 & 2.316 \\
\end{tabular}

  }
\end{table}

\newpage
\bibliographystyle{alpha}
\bibliography{nskbib,tristansbib,tristanswritings}
\end{document}